\documentclass[10pt,oneside,leqno]{amsart}
\usepackage{amsxtra}
\usepackage{amsopn}
\usepackage{amsmath,amsthm,amssymb}
\usepackage{amscd}
\usepackage{amsfonts}
\usepackage{color}
\usepackage{latexsym}
\usepackage{verbatim}
\usepackage{pb-diagram}
\usepackage[all]{xy}

\theoremstyle{plain}
\newtheorem{theorem}{Theorem}[section]
\newtheorem*{theorem*}{Theorem}
\newtheorem{lemma}[theorem]{Lemma}
\newtheorem{prop}[theorem]{Proposition}
\newtheorem{cor}[theorem]{Corollary}

\theoremstyle{definition}
\newtheorem{definition}[theorem]{Definition}
\newtheorem{ex}[theorem]{Example}

\newtheorem{rem}[theorem]{Remark}
\newtheorem{condition}[theorem]{Condition}

\newtheorem*{mt*}{Main Theorem}
\newcommand\C{{\mathbb C}}

\newcommand\g{{\mathfrak{g}}}

\newcommand\R{{\mathbb R}}

\newcommand\Z{{\mathbb Z}}

\newcommand\ad{{\rm ad}}
\newcommand\Ad{{\rm Ad}}
\newcommand\Aut{{\rm Aut}}
\newcommand\Der{{\rm Der}}
\newcommand\diag{{\rm diag\,}}
\newcommand\id{{\rm id}}
\newcommand{\cg}{{\mathfrak{g}}}

\newcommand{\cn}{{\mathfrak{n}}}

\begin{document}

\title[On   de Rham and Dolbeault  Cohomology of Solvmanifolds]{On   de Rham and Dolbeault  Cohomology of Solvmanifolds}

\author{Sergio Console, Anna Fino and  Hisashi Kasuya}
%\date{\today}
\address{(S. Console, A. Fino) Dipartimento di Matematica G. Peano \\ Universit\`a di Torino\\
Via Carlo Alberto 10\\
10123 Torino\\ Italy} \email{sergio.console@unito.it}
 \email{annamaria.fino@unito.it}
 \address{(H. Kasuya) Department of Mathematics, Tokyo Institute of Technology, 1-12-1, O-okayama, Meguro, Tokyo 152-8551, Japan}
  \email{khsc@ms.u-tokyo.ac.jp}
\thanks{This work was supported by the PRIN project ``Real and complex manifolds: geometry, topology and harmonic analysis'' and by GNSAGA
of INdAM}

\begin{abstract} 
For a simply connected  (non-nilpotent) solvable Lie group $G$ with a lattice $\Gamma$ 
the de Rham and Dolbeault cohomologies of the    solvmanifold $G/\Gamma$ are not in general isomorphic to the cohomologies of the Lie algebra $\g$ of $G$.
In this paper we  construct, up to a finite group, a new Lie algebra $\tilde \g$ whose cohomology is isomorphic to the de Rham cohomology of $G/\Gamma$ by using a modification of $G$ associated with a algebraic sub-torus of the Zariski-closure of the image of the adjoint representation.
This technique includes the construction due to Guan and developed by the first two authors.
In this paper, we also give a Dolbeault version of such technique for complex solvmanifolds, i.e. for  solvmanifolds endowed with an invariant complex structure.
We construct a finite dimensional cochain complex which computes the Dolbeault cohomology of  a complex  solvmanifold $G/\Gamma$ with holomorphic Mostow bundle and we give a construction of a new Lie algebra $\breve \g$ with a complex structure whose cohomology is isomorphic to the Dolbeault cohomology of $G/\Gamma$.
\end{abstract}

\maketitle

\vskip.4cm

\hskip7.8cm {\it To the  memory of Sergio Console}

\vskip.6cm

\section{Introduction}

Let $G/\Gamma$ be  a \emph{solvmanifold}, i.e.  a compact quotient of a connected and simply connected solvable Lie group $G$ by a lattice $\Gamma$. If $G$ is nilpotent, $G/\Gamma$ is called  a  \emph{nilmanifold}.

\smallskip
In \cite{Nomizu}, Nomizu showed that the de Rham cohomology of  a  nilmanifold $G/\Gamma$ is isomorphic to the cohomology of the Lie algebra $\cg$ of $G$.
The de Rham  cohomology of  a solvmanifold  $G/\Gamma$ is  not in general isomorphic to the   cohomology of the Lie algebra $\cg$ of $G$.  The isomorphism  holds if  $G$ 
is completely solvable \cite{Hattori}  and more  in  general  if  $G/\Gamma$  satisfies the \emph{Mostow condition},
i.e.  if  the algebraic closure  ${\mathcal A} (\Ad_G(G))$ of $\Ad_G (G)$ coincides with the algebraic closure   ${\mathcal A} (\Ad_G (\Gamma))$ of $\Ad_G (\Gamma)$ {(see \cite{Mostow2} and \cite[Corollary 7.29]{Raghunathan}).

In the general case, 
techniques for the computation of the de Rham cohomology of  compact solvmanifolds  were provided  in \cite{Guan,CF} and by the third author   in \cite{Kas, KDD}.

\begin{enumerate}

\item  The first two authors   in \cite{CF}, using  results by D. Witte \cite{Witte}  on the superigidity of lattices in solvable Lie groups, obtained   a different proof of a result of Guan  \cite{Guan},  which can be applied to compute the Betti numbers of a  compact solvmanifold
$G/ \Gamma$, even in the case that the solvable Lie group  $G$ and the lattice  $\Gamma$ do not satisfy the Mostow condition. The  basic idea of this method  is to modify the  solvable Lie group $G$ into a new  solvable Lie group  $\tilde G$ (diffeomorphic to $G$)  and possibly considering a finite index subgroup $\tilde \Gamma$ of $\Gamma$, in such a way that the new compact solvmanifold  $\tilde G/ \tilde \Gamma$ satisfies the Mostow condition (cf. Theorem~\ref{Mod11} here).

\item The third author constructed the explicit finite dimensional differential graded algebra (shortly DGA) $A_{\Gamma}^{\ast}$  which computes the de Rham cohomology of a solvmanifold $G/\Gamma$.
The finite dimensional cochain complex
$A_{\Gamma}^{\ast}$ is associated with the diagonal representation $\Ad_s: G \rightarrow \Aut(\frak g)
$ which is defined by the decomposition of the Lie algebra $\frak g$  as  in \cite[Proposition III.1.1]{DER}. (cf. Theorem~\ref{KKK} here).
\end{enumerate}

\smallskip

One of the purposes of this paper is  to unify these two approaches for the computation of the de Rham cohomology of $G/\Gamma$.   To do this,  we describe a general modification with respect to an algebraic subtorus $S$ of the maximal torus $T = \mathcal A (\Ad_s (G))$ of $\mathcal A (\Ad_G(G))$, which includes as a particular case the Witte modification (cf. Proposition~\ref{prop-mod}).
We extend the results  in \cite{Kas, KDD} and the modification method in \cite{CF} (see Remark~\ref{compare} and  Examples~\ref{ex-mod} and \ref{ex-mod-NEW}). 

\smallskip

Next, we  turn  to solvmanifolds  $G/\Gamma$ endowed with an invariant complex structure $J$ (i.e., with a complex structure induced by a left-invariant one on $G$). We will call $(G/\Gamma, J)$ a \emph{complex solvmanifold}.
We are interested in computing the Dolbeault cohomology of a complex solvmanifold $(G/\Gamma, J)$, providing  Dolbeault analogues of the above arguments  for de Rham cohomology.

For a  complex solvmanifold $(G/\Gamma, J)$, we consider the relation between the Dolbeault cohomology of $(G/\Gamma, J)$ and the Dolbeault cohomology of the Lie algebra $\frak g$ associated with a complex structure $J$ as in  \cite{CF0, CFGU}.
We say that a  complex solvmanifold $(G/\Gamma, J)$ admits a \emph{good Dolbeault cohomology} if  the  
 Dolbeault cohomology of a complex solvmanifold $(G/\Gamma, J)$ is isomorphic to the Dolbeault cohomology of the associated Lie algebra $\frak g$ with a complex structure $J$

Suppose $G$ is nilpotent.
Then, similar to  the Nomizu's theorem, the Dolbeault cohomology of a complex nilmanifold $(G/\Gamma, J)$ is isomorphic to the Dolbeault cohomology of the Lie algebra $\frak g$ associated with a complex structure $J$
if one of the following conditions holds:
  \begin{itemize}
    \item $(G,J)$ is a complex Lie group (\cite{Sakane});
    \item $J$ is an abelian complex structure (\cite{CF0});
    \item $J$ is a nilpotent complex structure (\cite{CFGU})  ;
    \item $J$ is a rational complex structure (\cite{CF0}).
\end{itemize}
It is conjectured that every nilmanifold with
an invariant  complex structure admits a good Dolbeault cohomology.

A complex solvmanifold $(G/\Gamma, J)$ does not admit a good Dolbeault cohomology even if $(G/\Gamma, J)$ is complex parallelizable (i.e. a compact quotient  of  a complex solvable Lie group).
Steps towards the computation of the Dolbeault cohomology of a complex solvmanifold $(G/\Gamma, J)$ were  previously taken by the third author in \cite{Kas, KDD}. In particular,  a finite dimensional cochain complex which allows to compute the Dolbeault cohomology of complex parallelizable solvmanifolds was given.
In this paper, we generalize these results constructing a differential bigraded algebra (DBA) which allows to compute the Dolbeault cohomology in more general situations (Theorem \ref{dolbb}).
Moreover,  in such situations, we provide  
Dolbeault analogues of our results, unifying the Console-Fino technique and the Kasuya technique on de Rham cohomology (Theorem \ref{semidirect_prod}).
In particular, we show that for a complex parallelizable solvmanifold $(G/\Gamma, J)$ (i.e. $(G,J)$ is a complex Lie group), there exists a finite index subgroup $\tilde \Gamma$ of $\Gamma$ and  a complex  Lie algebra $\breve\g$ such that the Dolbeault cohomology of a complex parallelizable solvmanifold $(G/\tilde \Gamma, J)$ is isomorphic to the Dolbeault cohomology of the complex Lie algebra $\breve\g$
(Corollary \ref{parall}).

\section{Preliminaries}

Solvmanifolds $G/\Gamma$ are determined up to diffeomorphisms by their fundamental groups (which coincide with their lattices $\Gamma$). This can be formulated by the following 

\begin{theorem} \cite[Theorem 3.6]{Raghunathan}  \label{diffeo} If $\Gamma_1$ and $\Gamma_2$ are lattices in simply connected solvable Lie groups $G_1$ and $G_2$,  respectively, and $\Gamma_1$ is isomorphic to $\Gamma_2$, then $G_1 / \Gamma_1$ is diffeomorphic to $G_2 / \Gamma_2$.
\end{theorem}

\subsection{Constructions of ${\rm Ad}_{s}$ and Kasuya's result} 
Let $\frak g$ be a solvable Lie algebra and $\frak n$ the nilradical of $\frak g$.  By Proposition III.1.1 in \cite{DER} there exists a vector space $V \cong \R^k$ such that ${\mathfrak g} = V \oplus \mathfrak n$ as direct sum of vector spaces and ${\mbox {ad}}(A)_s (B) =0$,  for any $A, B \in V$, where ${\mbox {ad}} (A)_s$  denotes  the semisimple part of ${\mbox {ad}}(A)$.

Like in \cite{Kas} one can define  the map 
$$
\ad_s: \frak g \rightarrow \Der (\frak g)
$$
by $\ad_s (A + X)(Y) =  (\ad_A)_s (Y)$, for $A \in V$ and $X \in \frak n$ and $Y\in \g$.

The map  $\ad_s$ is linear and 
$$
[\ad_s (\frak g), \ad_s (\frak g)] =0.
$$
Since the commutator $\frak g^1= [\frak g, \frak g]$ is contained in the nilradical $\frak n$ of $\frak g$, the map $\ad_s: \frak g \rightarrow \Der (\frak g)$ is a representation of $\frak g$ and the image $\ad_s (\frak g)$ is abelian and consists of semisimple elements. 

Note that, since $\ad (A)_s (B) =0$  for any $A, B \in V$, the vector space $V$ is a trivial sub-module of $\g$ which is a complement to $\cn$.

We will denote by $$\Ad_s: G \rightarrow \Aut(\frak g)
$$
the extension of the map $\ad_s$ to $G$. By the previous properties, it follows that $\Ad_s (G)$ is diagonalizable.

Let $T = {\mathcal A} (\Ad_s (G))$ be  the Zariski closure of $\Ad_s (G)$ in  $\Aut({\frak g}_{\C})$. Then $T$ is diagonalizable and it is a torus  in ${\mathcal A} (\Ad_G (G))$.

As  in \cite[Proof of Proposition 3.2]{Kas}, we have the following lemma.
\begin{lemma}  The Zariski closure $T = {\mathcal A} (\Ad_s (G))$ of $\Ad_s (G)$  is a  maximal torus of the Zariski closure  ${\mathcal A} (\Ad_G (G))$ of $\Ad_G (G)$. Moreover, ${\mathcal A} (\Ad_G (G)) = T\ltimes U$, where $U$  is  the unipotent radical of ${\mathcal A} (\Ad_G (G))$ and  
we have  $\Ad_s=p\circ \Ad$,  with   $p:{\mathcal A} (\Ad_G (G)) = T \ltimes U  \to T $  the projection on the maximal torus $T$.
\end{lemma}
\begin{proof}
It is known that  the map
$$
\phi: V \oplus \frak n \rightarrow G, A + X \mapsto \exp(A) \exp (X)
$$
is a global diffeomorphism (see \cite[Lemma 3.3]{Dek}).
For  $A\in V$, we consider   the Jordan decomposition ${\rm Ad}(\exp (A))=\exp ({\rm ad} (A)_{s}) \exp ({\rm ad} (A)_{n})$ where ${\rm ad} (A)_{n} $ is the nilpotent part of ${\rm ad} (A)$.
As  in \cite{Borel}, we have $\exp ({\rm ad} (A)_{s}),  \exp ({\rm ad} (A)_{n})\in {\mathcal A} (\Ad_G (G))$.
For $g=\exp(A) \exp (X)\in G$, we get
\[\Ad_{g}=\exp ({\rm ad} (A)_{s}) \exp ({\rm ad} (A)_{n})\exp  (X)\]
and here  $\exp({\rm ad} (A)_{s})=(\Ad_{s})_g$.
Let $U$ be the unipotent radical of ${\mathcal A} (\Ad_G (G))$.
Then we get  $\exp ({\rm ad} (A)_{s})=(\Ad_{s})_g\in T$ and $\exp ({\rm ad} (A)_{n})\exp  (X)\in U$.
Hence we have $\Ad_{g}\in T\cdot U$ and so $\Ad_G (G)\subset T\cdot U\subset {\mathcal A} (\Ad_G (G))$.
Thus $T\cdot U={\mathcal A} (\Ad_G (G))$ and $T$ is a maximal torus of ${\mathcal A} (\Ad_G (G))$.
As a consequence we obtain $T\cdot U=T\ltimes U$ and for the projection $p:T\ltimes U\to T$, we obtain
\[p(\Ad_{g})=p(\exp ({\rm ad} (A)_{s}) \exp ({\rm ad (A)}_{n})\exp  (X))=\exp ({\rm ad} (A)_{s})=(\Ad_{s})_g.
\]
Hence the lemma follows.
\end{proof}

Let $G$ be a simply connected solvable Lie group with a  lattice $\Gamma$  and $\frak g$ the  Lie algebra of $G$. Take a basis $X_1, \ldots, X_n$ of the complexification $\frak g_{\C}$ of $\frak g$ such that $\Ad_s$ is represented by diagonal matrices as
$$
({\Ad_s})_g = \diag (\alpha_1 (g), \ldots, \alpha_n (g))\, , 
$$
for any $g \in G$, i.e. $({\Ad_s})_g X_i = \alpha_i (g) X_i$ for  the characters $\alpha_i$ of $G$. Let $x_1, \ldots, x_n$ be the dual basis of $X_1, \ldots, X_n$.

We consider the sub-DGA $A^{\ast}_{\Gamma}$ of the de Rham complex $A^{\ast}(G/\Gamma)\otimes \C$ which is given by 
\[
A^{p}_{\Gamma}
={\rm span} \left\langle \alpha_{I} x_{I} {\Big \vert} \begin{array}{cc}I\subset \{1,\dots,n\},\\  (\alpha_{I})_{\vert_{\Gamma}}=1 \end{array}\right\rangle.
\]
where
for a multi-index $I=\{i_{1},\dots ,i_{p}\}$ we write $x_{I}=x_{i_{1}}\wedge\dots \wedge x_{i_{p}}$,  and $\alpha_{I}=\alpha_{i_{1}}\cdots \alpha_{i_{p}}$.

\begin{theorem}[\cite{Kas}]\label{KKK}
The inclusion 
\[A_{\Gamma}^{\ast}\subset A^{\ast}(G/\Gamma)\otimes \C
\]
induces a cohomology isomorphism.
\end{theorem}

Since the characters $\left\{ \alpha_{I} {\Big \vert} \begin{array}{cc}I\subset \{1,\dots,n\},\\  (\alpha_{I})_{\vert_{\Gamma}}=1 \end{array}\right\}$ are obstructions for the  invariance  of elements of the DGA $A_{\Gamma}^{\ast}$, we have the following theorem.

\begin{theorem}[\cite{Kas}] Let $G$ be a simply connected solvable Lie group with a  lattice $\Gamma$  and $\frak g$ the  Lie algebra of $G$. Suppose that the following condition holds: \newline for any $I \subset \{ 1, \ldots, n \}$   if  the  product  $\alpha_{I}$ is non-trivial, then the restriction of $\alpha_{I} \vert_{\Gamma}$ in $\Gamma$ is also non-trivial. \newline Then the isomorphism
$$
H^* (G/\Gamma, \C) \cong H^* (\frak g_{\C})
$$
holds.
\end{theorem}

As a consequence, we obtain the isomorphism $H^*_{dR} (M) \cong H^* (\frak g)$.

\begin{rem} There exist  solvmanifolds which satisfy the assumption of the previous theorem, but $G$ and $\Gamma$ do not satisfy the Mostow condition (see for instance Example 3 in \cite{Kas}).
\end{rem}

\subsection{Nilshadows and  Witte modification}\label{nilwitte}

When $G$ and $\Gamma$ do not satisfy the Mostow condition, one may apply the Witte modification \cite{Witte}, which is a variation of  the construction of the nilshadow due to  Auslander and Tolimieri \cite{AT}.  
In general, one has that 
${\mathcal A} (\Ad_G (G)) = T \ltimes U$ is not unipotent, where $T$ is a non trivial  maximal torus of  ${\mathcal A} (\Ad_G (G))$. Moreover, we have $T = T_{split} \times T_{cpt}$,  where by $T_{split}$ we denote the maximal $\R$-split subtorus of $T$ and  by   $T_{cpt}$  the maximal compact subgroup of $T$.  The basic idea  for the construction of  the nilshadow is to kill $T$ in order to obtain  a nilpotent group. 
In order to do this we  define, following \cite{Witte},  a natural homomorphism $\pi: G \to T$, which is the composition of the homomorphisms:
 \begin{equation}\label{*} \pi: G \stackrel{\Ad} \to {\mathcal A} (\Ad_G (G)) \stackrel{{\mbox{\tiny projection}}}  \longrightarrow T,
\end{equation}
 and the map
$$
 \Delta: G \to T \ltimes G,  g \mapsto \Delta(g) = (\pi(g)^{-1}, g).
$$
The traditional nilshadow construction kills the entire maximal torus $T$ in order to get the nilpotent Lie group ${\Delta} (G)$.
It is known that ${\Delta} (G)$ is the nilradical of $ T \ltimes G$,  which is called the \emph{ nilshadow} of $G$.

 Witte  introduced in \cite{Witte} a variation of the nilshadow construction, killing only a subtorus $S$ of $T_{cpt}$.
It is well known that, for every subtorus $S$ of a compact torus $T$, there is a torus $S^{\perp}$ complementary to $S$ in $T$, i.e. such that $T = S \times S^{\perp}$.

As a consequence  of \cite[Proposition 8.2]{Witte}   the following  was proved

\begin{theorem}\label{Mod11} \cite{Guan,CF}  Let $M = G / \Gamma$ be a  solvmanifold,  compact  quotient of  a simply connected solvable Lie group $G$  by a lattice  $\Gamma$,  and let $T_{cpt}$ be a   compact torus such that  $$T_{cpt}  {\mathcal A} (\Ad_G (\Gamma)) = {\mathcal A} (\Ad_G (G)).$$ Then there exists a  subgroup $\tilde \Gamma$ of finite index in $\Gamma$ and a  simply connected   normal subgroup $\tilde G$ of $T_{cpt} \ltimes G$  such that  $${\mathcal A} (\Ad_{\tilde G}  (\tilde \Gamma)) = {\mathcal A} (\Ad_{\tilde G} ( \tilde G)).$$ Therefore, 
 $ \tilde G /  \tilde \Gamma$ is diffeomorphic to $G /  \tilde \Gamma$ and $H_{dR}^* (G /  \tilde \Gamma) \cong H^*(\tilde {\mathfrak g})$, where $\tilde {\mathfrak g}$ is the Lie algebra of $\tilde G$.

\end{theorem}

\smallskip

The idea of the proof  in \cite{CF}  is to construct  a   modification  $\tilde G$ of the solvable Lie group $G$ that we will now review shortly.  By \cite[Theorem 6.11, p. 93]{Raghunathan} we may consider a finite
    index subgroup   $\tilde \Gamma$ of $\Gamma$ such that  ${\mathcal A} ( \Ad_G (\tilde \Gamma))$ is connected. Let $T_{cpt}$ be a maximal  compact torus   of  ${\mathcal A} (\Ad_G (G))$  which contains a maximal compact torus $S_{cpt}^{\perp}$ of ${\mathcal A} (\Ad_G (\tilde \Gamma))$. There is  a natural
projection from ${\mathcal A} (\Ad_G (G))$ to $T_{cpt}$, given by the splitting ${\mathcal A} (\Ad_G (G)) = (A \times  T_{cpt}) \ltimes U,$ where $A$
is a maximal $\R$-split torus and  $U$ is the unipotent radical.

 Let $S_{cpt}$ be a subtorus of $T_{cpt}$ complementary to $S_{cpt}^{\perp}$ so that $T_{cpt} = S_{cpt} \times S_{cpt}^{\perp}$.
Let $\sigma$  be the composition of the homomorphisms:
$$
\sigma: G  \stackrel{\Ad}  \longrightarrow {\mathcal A} (\Ad_G (G))  \stackrel{{\mbox {\tiny projection}}}  \longrightarrow T_{cpt}  \stackrel{{\mbox {\tiny projection}}} \longrightarrow S_{cpt} \stackrel{ x \to x^{-1}}  \longrightarrow S_{cpt}. 
%\leqno{(*)}
$$
One may define the nilshadow map:
$$
\Delta: G \to S_{cpt} \ltimes G, g \mapsto (\sigma (g), g),
$$
which is not a homomorphism (unless $S_{cpt}$ is trivial and then $\sigma$ is trivial), but, with respect to the product on $G$, one has
$$
\Delta (ab) = \Delta ( \sigma(b^{-1})  a \sigma( b))  \,  \Delta (b),  \quad \forall a,b \in G
$$
and $\Delta (\gamma g) = \gamma \Delta (g)$, for every $\gamma \in \tilde \Gamma, g \in G$.
The nilshadow map $\Delta$ is a diffeomorphism from $G$ onto its image $\Delta (G)$ and then $\Delta (G)$ is simply connected. More explicitly, the product in $\Delta (G)$ is given by:
$$
\Delta (a) \Delta (b) = (\sigma (a), a)  \, (\sigma (b), b) = (\sigma (a) \sigma (b), \sigma (b^{-1})  a \sigma (b) \, b),
$$
for any $a, b \in G$.

By construction, ${\mathcal A} (\Ad_G (G))$ projects trivially  to $S_{cpt}$ and $\sigma ( \tilde \Gamma) = \{ e \}$. Therefore
$$\tilde \Gamma = \Delta ( \tilde \Gamma) \subset \Delta (G).
$$
Let $\tilde G = \Delta (G)$.
 By \cite[Proposition  4.10]{Witte}  $S_{cpt}^{\perp}$ is a maximal compact subgroup of ${\mathcal A} (\Ad_{\tilde G} (\tilde G))$ and $S_{cpt}^{\perp} \subset {\mathcal A} (\Ad_{\tilde G} (\tilde \Gamma))$, therefore  $\tilde G$ and $\tilde \Gamma$  satisfy the Mostow condition.

By using Theorem \ref{diffeo}, $G/\tilde \Gamma$ is diffeomorphic to $\tilde G / \tilde \Gamma$. Since $\tilde G$ and $\tilde \Gamma$  satisfy the Mostow condition, $H^*(\tilde G / \tilde \Gamma) \cong H^*(\tilde {\mathfrak g})$, so  
$H^*(G / \tilde \Gamma) \cong H^*(\tilde {\mathfrak g})$ and the inverse $\Delta^{-1}$ of the diffeomorphism 
$$
\Delta: G \to \tilde G,
$$
 induces a finite-to-one covering
map  $\Delta^*: \tilde G/ \tilde \Gamma \to G/ \Gamma$.

 \section{Modified solvable Lie groups and  $\Ad_s$-invariant geometry}\label{modif}
 
In this section we will define a modification of a  simply connected real solvable Lie group $G$ with respect to a sub-algebraic torus $S$ of the maximal torus $T = \mathcal A (\Ad_s (G))$, which extends the Witte modification $\Delta (G)=\tilde G$ described in the previous Section.

 \subsection{$S$-modification for any sub-algebraic torus $S\subset T$}
\begin{definition}\label{modified_product}
Let $G$ be a simply connected solvable Lie group.
Consider the $\C$-diagonalizable representation $\Ad_{s}:G\to \Aut(\g)$ and the Zariski-closure $T$ of $\Ad_{s}(G)$ in  $\Aut(\g)$.
Let $S\subset T$ be a sub-algebraic torus.
Consider the homomorphism
$\pi_{S}:G \stackrel{\pi} \to T\to S$ (where $\pi$ is the map in \eqref{*}).

Then we define the new product $\bullet_{S}$ on $G$ by
 \begin{equation}  \label{prodmodification} a\bullet_{S} b=a \pi_{S}(a)^{-1}(b) \end{equation}
for $a, b\in G$.
We denote by $G^{S}$ the Lie group $G$ endowed with the product $\bullet_S$ and call $G^{S}$ the $S$-modification of $G$.
\end{definition}

Note that, by definition, $G^S$ is still solvable.
By using Definition~\ref{modified_product}  and the  results in  \cite{Auslander}  we can show the following properties.

\begin{prop}\label{prop-mod} Let $G$ be a simply connected solvable Lie group and let $G^S$  the $S$-modification of $G$, with $S$ a sub-algebraic torus  of $T$. Then 

\begin{enumerate}

\item  If $S=T$, then  $G^{S}$ is the nilshadow of $G$.

\item  If $S=T_{cpt}$ (resp $T_{split}$), then $G^{S}$ is completely solvable (resp. of (I)-type).

\item  Suppose that $G$ has a lattice $\Gamma$.
If the restriction $\pi_{S}\vert_{\Gamma}$ is trivial, then $\Gamma$ is also a lattice in $G^{S}$ and hence we have a diffeomorphism $G/\Gamma\cong G^{S}/\Gamma$.

\item  
Suppose that  $G$ has a lattice $\Gamma$ and that  the Zariski-closure of $\Ad_{s}(\Gamma)$ is connected.
Take $S$ a subtorus complementary to a maximal compact torus of the Zariski-closure of $\Ad_{s}(\Gamma)$.
Then $G^{S}$ and $\Gamma$ satisfy  the Mostow condition.

\item  Suppose  that $G$ has a lattice $\Gamma$.
Then $G$ is of (I)-type if and only if there exists   a finite index subgroup $\tilde \Gamma$ of $\Gamma$ such that  the restriction $\Ad_{s}\vert_{\tilde\Gamma}$ is trivial.

\end{enumerate}

\end{prop}
\begin{proof}
\begin{enumerate}
\item  If  $S=T$,  we can easily check that $G^{S}$ is isomorphic to the nilshadow  that  we constructed in Section  \ref{nilwitte}.
\item If  $S=T_{cpt}$ (resp $T_{split}$), for the modified Lie group $G^{S}$, considering the map ${\rm Ad}_{s}: G^{S}\to \Aut(\g^{S})$,
 we can check ${\rm Ad}_{s}(G^{S})\subset T_{split}$ (resp. ${\rm Ad}_{s}(G^{S})\subset  T_{cpt}$)
 and hence  $G^{S}$ is completely solvable (resp. of (I)-type).
 \item The restriction  of  the product $\bullet_S$ on $\Gamma$ is not changed and hence  $\Gamma$ is a subgroup of $G^{S}$.
Obviously, $\Gamma$ is discrete and cocompact in $G^{S}$.
 \item See  Section  \ref{nilwitte} and \cite{CF}.
 \item  Note that the (I)-type  corresponds to the  \lq \lq type R" in \cite{Auslander}.
Denote by $N_G$ the nilshadow of $G$.
We have $T\ltimes G=T\ltimes N_G$.
If we suppose that $G$ is of (I)-type, then by the argument in
\cite[Chapter IV. 5.]{Auslander} and the inclusion $G\subset T\ltimes N_G$, we get that
$\Gamma\cap N_G$ is a finite index subgroup of $\Gamma$.
Take $\tilde \Gamma =\Gamma\cap N_G$.
Then the composition 
\[  \tilde \Gamma \stackrel{{\mbox{\tiny inclusion}}}  \to T\ltimes G= T\ltimes  N_G \stackrel{{\mbox{\tiny projection}}}  \longrightarrow T
\]
is trivial.
Now, this map is  $\Ad_{s}:\Gamma\to T$.
Hence  the restriction $\Ad_{s}\vert_{\tilde\Gamma}$ is trivial.
Conversely, suppose there exist a finite index subgroup $\tilde \Gamma$ of $\Gamma$ such that the restriction $\Ad_{s}\vert_{\tilde\Gamma}$ is trivial.
By \cite[Theorem 3.28]{Witte}, for the Zariski-closure $T_{\Gamma}$ of ${\rm Ad}_{s}(\Gamma)$ in $\Aut(\g)$, 
we have $T=T_{\Gamma}\cdot T_{cpt}$.
Therefore,  since $\Ad_{s}\vert_{\tilde\Gamma}$ is trivial, $T= T_{cpt}$ and hence $G$ is  of (I)-type.
\end{enumerate}

\end{proof}

\subsection{Modification and invariant cohomology}

Let $\g$ be the Lie algebra  of $G$ and $\g^{S}$ the Lie algebra of the $S$-modification $G^{S}$  with product $\bullet_{S}$ defined in \eqref{prodmodification}.

Denote by $L^{G}_{g}$ the left translation of $g\in G$ on $G$  and  by 
$L^{G^{S}}_{g}$ the left translation of $g$ in $G^S=(G, \bullet_S)$.
Then  $L^{G^{S}}_{g}=L^{G}_{g}\circ \pi_{S}(g)^{-1}$, for every $g \in G$.
Hence 
\[\g^{S}=\{\tilde X \in C^{\infty}(TG)\vert X\in T_{e}G,\; \; (\tilde X)_{g}=L^{G}_{g}\circ  \pi_{S}(g)^{-1} X\}.
\]

Since $\Ad_{s}: G\to \Aut(\g)$ is $\C$-diagonalizable, there exists a basis $X_{1},\dots,X_{n}$ of $\cg_\C=\g\otimes \C$ such that $\Ad_{s}=\diag(\alpha_{1},\dots,\alpha_{n})$ for some characters $\alpha_{1},\dots,\alpha_{n}$.
We also have $\pi_{S}=\diag(\beta_{1},\dots,\beta_{n})$  for some characters $\beta_{1},\dots,\beta_{n}$.
Then, since
 \[\g^{S}=\{\tilde X \in C^{\infty}(TG)\vert X\in T_{e}G,\; \; (\tilde X)_{g}=L^{G}_{g}\circ  \pi_{S}(g)^{-1} X\},
\]
we obtain that $\beta^{-1}_{1} X_{1},\dots , \beta^{-1}_{n} X_{n}$ is a basis of $\g^{S}\otimes \C=\cg^S_\C$.
Thus we have
 $$
 \frak g^S_{\C} = {\rm {span}} < \beta_1^{-1} X_1, \ldots, \beta_n^{-1} X_n>
 $$
 and 
 $$
 {\bigwedge}^*  (\frak g^S)^*_{\C} = {\bigwedge}^*  < \beta_1 x_1, \ldots, \beta_n x_n>,
 $$
 where $x_1, \ldots, x_n$ is the dual basis of $X_1, \ldots, X_n$.

 If $S = T$ we get
 $$
{\bigwedge}^* {\frak u}_{\C}^* = {\bigwedge}^* < \alpha_1 x_1, \ldots,  \alpha_n x_n>,
 $$
 where ${\frak u}$ is the Lie algebra of the nilshadow of $G$.
 
 Let $$A^*_{\Gamma} = A^* (G/\Gamma) \cap {\bigwedge}^* {\frak u}_{\C}^*,$$  where  $A^k (G/\Gamma)$ denotes the space of $k$-forms on $G/\Gamma$. Then
 $$
 A^*_{\Gamma} = {\rm {span}} \left\langle \alpha_{I} x_{I} {\Big \vert} \begin{array}{cc}I\subset \{1,\dots,n\},\\  (\alpha_{I})_{\vert_{\Gamma}}=1 \end{array}\right\rangle
 $$
  By  Theorem \ref{KKK},  the inclusion $A^*_{\Gamma}\subset A^* (G/\Gamma)$ induces  a cohomology isomorphism.
Hence we have:

 \begin{theorem}\label{rel_cohom}  Let $G/\Gamma$ be a solvmanifold.  If there exists a subtorus $S$ of $T$ such that $\pi_S (\Gamma) = 1$ and $A^*_{\Gamma}  \subset  \bigwedge  ({\frak g}^S)^*_{\C}$, then $H^*_{dR} (G /\Gamma) \cong H^* (\frak g^S).$
\end{theorem}

\begin{rem}\label{compare}
Consider a finite index subgroup  $\tilde\Gamma\subset \Gamma$ such that $S^{\bot} ={\mathcal A}(\Ad_{s}(\tilde\Gamma))$ (and hence  ${\mathcal A}(\Ad (\tilde\Gamma))$)  is connected.
We take a sub-torus $S\subset T$ such that $T=S\times S^{\bot}$.
Then we have  $A^*_{\tilde\Gamma}  \subset    {\bigwedge}^*  (\g^S)^*_{\C}$.
Hence Theorem \ref{rel_cohom} is a generalization of Theorem \ref{Mod11}.
\end{rem}
%\begin{proof}

\smallskip

 Indeed, consider the projections $p_{S}:T\to S$ and $p_{S^{\bot}}:T\to S^{\bot}$. Observe that $\pi_{S}=p_{S}\circ \Ad_{s}$. Let   $D_{n}(\C)$ be the set of the  complex $n \times n$ diagonal matrices. 
For \[t_{i}\in {\rm {Char}} (D_{n}(\C))=\{t_{1}^{m_{1}}\cdots t_{n}^{m_{n}}\vert (t_{1},\dots, t_{n})\in D_{n}(\C)\},\]
we define $f_{i}=(t_{i})_{\vert T}$, $g_{i}= (t_{i})_{\vert S}\circ p_{S}$ and $h_{i}=(t_{i})_{\vert S^{\bot}} \circ p_{S^{\bot}}$.
We have $f_{i}\circ \Ad_{s}=\beta_{i}$, $g_{i}\circ \Ad_{s}=\beta_{i}$ and  $f_{i}=g_{i}h_{i}$.
Suppose $(\alpha_{I})\lfloor_{\Gamma}=1$ for some $I \subseteq \{1,\ldots,n\}$ and consider 
\[f_{I}=g_{I}h_{I}.
\]
Then, since $\Ad_{s}(\tilde\Gamma)\subset S^{\bot}$, we get $(g_{I})_{\vert \Ad_{s}(\tilde \Gamma)}=1$. Hence
\[1=\left(\alpha_{I}\right)\lfloor_{\tilde \Gamma}=(f_{I})_{\Ad_{s}(\tilde \Gamma})=(h_{I})_{\Ad_{s}(\tilde \Gamma}).
\]
Since $S^{\bot}$ is Zariski-closure of $\Ad_{s}(\tilde \Gamma)$, we obtain $h_{I}=1$ and thus $f_{I}=g_{I}$.
Hence if  $(\alpha_{I})\lfloor_{\Gamma}=1$ for  $I \subseteq \{1,\ldots,n\}$, we have 
\[\alpha_{I}=\beta_{I}.
\]
This implies  $A^*_{\tilde\Gamma}  \subset  \bigwedge  (\g^S)^*_{\C}$ and gives a proof of the above remark.
%\end{proof}

\begin{ex}\label{ex-mod}
Let $G=\R\ltimes_{\phi} \R^{2}$ such that
$\phi(t)=\left(\begin{array}{cc} \cos \pi t & -\sin \pi t \\ \sin \pi t &\cos \pi t\end{array}\right)$.
Then $G$ has two lattices $\Gamma_{1}=\Z\ltimes \Z^{2}$ and $\Gamma_{2}=2\Z\ltimes \Z^{2}$.
In the case of $\Gamma_{1}$,  we take $S$ trivial.
Then $A^*_{\Gamma}  \subset  \bigwedge  (\g^S)^*_{\C}$ and hence $H^{\ast}(G/\Gamma_{1})\cong H^{\ast}(\g^S)\cong H^{\ast}(\g)$.
In the case of $\Gamma_{2}$, we take $S=T=\Ad_{s} (G)$.
Then, $A^*_{\Gamma}  \subset  \bigwedge  (\g^S)^*_{\C}$ and hence $H^{\ast}(G/\Gamma_{2})\cong H^{\ast}(\g^S)\cong H^{\ast}(\R^{3})$.
\end{ex}

\begin{ex}\label{ex-mod-NEW}
Let $G=\R\ltimes_{\phi} \R^{4}$ such that
\[\phi(t)=\left(\begin{array}{cccc} \cos \pi t & -\sin \pi t & t\cos \pi t & -t\sin \pi t
\\ \sin \pi t &\cos \pi t&t\sin \pi t &t\cos \pi t\\
0&0&\cos \pi t & -\sin \pi t\\
0&0&\sin \pi t &\cos \pi t
\end{array}
\right).\]
Then we have the Jordan decomposition
\[\phi(t)=\left(\begin{array}{cccc} \cos \pi t & -\sin \pi t & 0 & 0
\\ \sin \pi t &\cos \pi t&0&0\\
0&0&\cos \pi t & -\sin \pi t\\
0&0&\sin \pi t &\cos \pi t
\end{array}
\right)\left(\begin{array}{cccc} 1 & 0 & t & 0
\\ 0&1&0&t\\
0&0&1 & 0\\
0&0&0 &1
\end{array}
\right).
\]
Hence the maximal torus  $\mathcal A (\Ad_s (G))$ is $S^{1}$ and the nilshadow $N_G$  of $G$ is $\R\ltimes_{\tilde\phi} \R^{4}$ where
\[\tilde\phi(t)=
\left(\begin{array}{cccc} 1 & 0 & t & 0
\\ 0&1&0&t\\
0&0&1 & 0\\
0&0&0 &1
\end{array}
\right).\]
Take $\Gamma=2\Z\ltimes \Z^{4}$.
Then $\Gamma$ is lattice in $G$ and $\Ad_{s}\vert_{\Gamma}$ is trivial.
In this case, the complement of $\mathcal A (\Ad_s (\Gamma))$ in $T$ is $S^{1}$.
Hence, for the $S^{1}$-modification $G^{S^{1}}$, we have an isomorphism
\[H^{\ast}(G/\Gamma)\cong H^{\ast}(\g^{S^1}),\]
where  $\g^{S^1}=span<X_{1},X_{2},X_{3},X_{4},X_{5}>$
with non-zero  Lie brackets
$$ [X_{1}, X_{4}]=X_{2},  \quad [X_{1},X_{5}]=X_{3}.
$$
By $T=S^{1}$, $G^{S^{1}}$ is the nilshadow  $N_G$ and so the cohomology $H^{\ast}(G/\Gamma)$ is computed by the Lie algebra of $N_G=\R\ltimes_{\tilde\phi} \R^{4}$.

\end{ex}

\subsection{Modification and invariant complex structures}
Let $J$ be a left-invariant complex structure on $G$ such that $J\circ (\Ad_{s})_g=(\Ad_{s})_g\circ J$ for every $g\in G$.
Since $T$ is the Zariski closure of $\Ad_{s}(G)$, $t$ and $J$ commute, for every $t\in T$.
Moreover, for any $\tilde X \in \g^{S}$ we have
\[J_{g}(\tilde X)_{g}=J_{g}\circ L^{G}_{g}\circ  \pi_{S}(g)^{-1} X= L^{G}_{g}\circ  \pi_{S}(g)^{-1}( J_{e} X)=L^{G^{S}}_{g}(J_{e}(\tilde X)_{e}). \]
As a consequence
\begin{prop} \label{cxstrmodif}
Let $J$ be a left-invariant complex structure on $G$ such that $J\circ (\Ad_{s})_g=(\Ad_{s})_g\circ J$ for every $g\in G$.
Then $J$ can be considered as a left-invariant complex structure on the $S$-modification $G^{S}=(G, \bullet_S)$.
Moreover for a lattice $\Gamma$ of $G$ such that the restriction $\pi_{S}\vert_{\Gamma}$ is trivial,  the complex solvmanifold $(G/\Gamma,J)$ is biholomorphic to $G^{S}/\Gamma$ endowed with the invariant complex structure induced by the left-invariant complex structure $J$ on $G^{S}$.
\end{prop}

We can apply Proposition~\ref{cxstrmodif} to the following two examples.

\begin{ex} $(${\bf Nakamura manifold}$)$ \label{naka}  Consider the  simply connected complex solvable  Lie group $G$
 defined by 
$$
G = \left \{ \left(  \begin{array}{cccc} e^z&0&0&w_1\\ 0&e^{-z}&0&w_2\\ 0&0&1&z\\ 0&0&0&1 \end{array} \right) , \,  w_1, w_2, z \in \C \right \}.
$$
The Lie group $G$ is  the semi-direct product $\C  \ltimes_{\varphi} \C^2$, where 
$$
\varphi (z) = \left(  \begin{array}{cc} e^z&0\\ 0&e^{-z} \end{array} \right) \, , \qquad \forall z \in \C\, .
$$
A basis of  complex left-invariant $1$-forms is given by
$$
\phi_1 = dz, \, \phi_2 = e^{- z} d w_1, \, \phi_3 = e^z d w_2
$$
and in terms of the  real basis of left-invariant $1$-forms $(e^1, \ldots, e^6)$ defined by
$$
\phi_1 = e^1 + \sqrt{-1}  e^2, \, \phi_2 = e^3 + \sqrt{-1} e^4, \, \phi_3 = e^5 + \sqrt{-1}  e^6,
$$
the  structure equations are:
$$
\left \{ \begin{array} {l}
d e^j = 0, \quad  j = 1,2,\\[4 pt]
d e^3 = - e^{13} + e^{24},\\[4 pt]
d e^4 = - e^{14} - e^{23},\\[4 pt]
d e^5 = e ^{15} - e^{26},\\[4 pt]
d e^6 = e^{16} + e^{25},
\end{array} \right .
$$
where  we denote by $e^{ij}$ the wedge product  $e^i \wedge e^j$.

Let $B \in SL(2, \Z)$ be a unimodular matrix with distinct real eigenvalues: $\lambda, \frac {1} {\lambda}$.
If  $t_0 = \log \lambda$, i.e. $e^{t_0} = \lambda$,
then there exists a matrix $P \in GL(2, \R)$ such 
that $$
P B P^{-1} =  {\mbox{diag}  (\lambda, \lambda^{-1})}.
$$
Let 
$$
\begin{array}{l}
L_{1, 2 \pi} = \Z [t_0, 2 \pi i] = \{ t_0 k + 2 \pi h i  \,  \, \vert \, \,   h, k \in \Z \},\\ [4pt]
L_2 = \left\{  P \left(  \begin{array}{c}  \mu \\ \alpha \end{array} 
 \right) \, \,  \vert \, \,   \mu,  \alpha \in \Z[i] 
\right \}.
\end{array}
$$
 Then, by \cite{Yamada}  $\Gamma = L_{1, 2 \pi} \ltimes_{\varphi} L_2$  is a lattice of $G$.
 
 Since $G$ has trivial center, we have that $\Ad_G (G) \cong G$ and thus $\Ad_G (G)$ is a semidirect product $\R^2 \ltimes \R^4$. Moreover,
for  the Zariski closures of $\Ad_G (G)$ and $\Ad_G (\Gamma)$ we obtain $$
 \begin{array}{l}
 {\mathcal A} (\Ad_G (G)) = (\R^{\#}  \times S^1) \ltimes  \R^4,\\
 {\mathcal A} (\Ad_G (\Gamma)) = \R^{\#}  \ltimes  \R^4,\\
\end{array} 
$$
where  the split torus $\R^{\#}$ corresponds to the action of $e^{\frac{1}{2} (z + \overline z)}$ and the compact torus $S^1$ to the one of $e^{\frac{1}{2} (z - \overline z)}$.

Therefore in this case ${\mathcal A} (\Ad_G (G))  = S^1  {\mathcal A} (\Ad_G (\Gamma))$ and ${\mathcal A} (\Ad_G (\Gamma))$ is connected.
By applying Theorem~\ref{Mod11} there exists a simply connected normal subgroup $\tilde G = \Delta (G)$ of  $S^1 \ltimes G$. The new Lie group $\tilde G$ is obtained  by killing the action of $e^{\frac{1}{2} (z - \overline z)}$. Indeed, 
we get  that 
$$
\tilde G  \cong  \left \{ \left(  \begin{array}{cccc} e^{\frac{1}{2} (z + \overline z)}&0&0&w_1\\ 0&e^{- \frac{1}{2} (z + \overline z)}&0&w_2\\ 0&0&1&z\\ 0&0&0&1 \end{array} \right), w_1, w_2, z \in \C \right \}.
$$
Note that $\tilde G$ is the modification $G^S$ of $G$,  where the algebraic subtorus   $S$  of $T  = {\mathcal A} (\Ad_s  (G)) \cong \R^{\#} \times S^1$  corresponds  to the action of $e^{\frac{1}{2} (z - \overline z)}$.

The diffeomorphism between $G/\Gamma$ and $\tilde G/\Gamma$ was already shown in \cite{Yamada}.  Then in this case  one has the isomorphism $H^*_{dR} (G/ \Gamma)  \cong H^* (\tilde {\mathfrak g})$, where $\tilde {\mathfrak g}$ denotes the Lie algebra of $\tilde G$ and the de Rham cohomology of the Nakamura manifold  $G/\Gamma$ is not isomorphic to $H^* ({\mathfrak g})$ (see also \cite{debaT}).
The Nakamura manifold $G/\Gamma$ is a complex parallelizable manifold and it is endowed with the bi-invariant complex structure $J$ with $(1,0)$-forms $(\phi_1, \phi_2, \phi_3)$.  In particular we have
$
J \circ (\Ad_{s})_g = (\Ad_{s})_g \circ J,$ for every $g \in G$. Applying Proposition \ref{cxstrmodif} $J$ can be viewed as a left-invariant complex structure on the $S$-modification $G^S =\tilde G$ and since the restriction $\pi_S \vert_{\Gamma}$ is trivial, the Nakamura manifold $(G/\Gamma, J)$ is biholomorphic to $(G^S/\Gamma, J)$. This property was already proved in \cite{Yamada}.

\end{ex}

\begin{ex} $(${\bf Secondary Kodaira surface}$)$ \label{naka}  Consider the   semi-direct product $G = \R \ltimes_{\rho} H_3 (\R)$, where 
$$
\rho (t) = \left(  \begin{array}{ccc} \cos t&-\sin t & 0\\ \sin t &\cos t & 0\\ 0 & 0& 1 \end{array} \right) \, , \qquad \forall t  \in \R, \, 
$$
and $H_3 (\R)$ is  the $3$-dimensional  real Heisenberg group  which is considered as $\R^{3}$ with the multiplication
$$(x,y,z)(x^{\prime},y^{\prime},z^{\prime})=(x+x^{\prime}, y+y^{\prime},z+z^{\prime}+\frac{1}{2}(xy^{\prime}-x^{\prime}y)).$$
$G$ admits a  basis of   left-invariant $1$-forms  $(e^1, \ldots, e^4)$ such that
$$
\left \{ \begin{array} {l}
d e^1= e^{24},\\[4pt]
d e^2= - e^{14} ,\\[4 pt]
d e^3 = - e^{12},\\[4 pt]
d e^4 = 0.
 \end{array}  \right.
$$
By \cite{COS} the Lie group $G$ admits lattices of the form $\Lambda_{k, l} = l \Z \ltimes \Gamma_k$, where $k \in \mathbb N$, $l = \pi,  \frac{\pi}{2}$ and  $\Gamma_k= \Z \times \Z \times \frac{1}{2k} \Z$. The  left-invariant complex structure on $G$  defined by 
$$
J e_1 = e_2, \quad J e_3 = e_4
$$
satisfies  the condition $J\circ (\Ad_{s})_g=(\Ad_{s})_g\circ J$, for every $g\in G$. Therefore, by applying Proposition \ref{cxstrmodif},  we have that $J$ can be viewed as a left-invariant complex structure on the $S$-modification $G^S \cong \R\times H_{3}(\R)$, where  $S = S^1$  is generated by $e_4$.
\end{ex}

\section{Modifications and   Holomorphic Mostow fibrations}\label{hol_Mostow}
Let $G$ be  a simply connected  solvable Lie group with a lattice $\Gamma$ and $\g$ be the Lie algebra of $G$.
Let $N$ be the nilradical of $G$.
It is known that $\Gamma\cap N$ is a lattice of $N$ and $\Gamma/\Gamma\cap N$ is a lattice of the abelian Lie group $G/N$ (see \cite{Raghunathan}).
The solvmanifold $G/\Gamma$ is a fiber bundle
\[\xymatrix{
N/\Gamma\cap N=N\Gamma/\Gamma\ar[r]& G/\Gamma\ar[r]  &G/N\Gamma=(G/N)/(\Gamma/\Gamma\cap N)
}
\]
 over a torus with a  nilmanifold $N/\Gamma\cap N$ as fiber. Here the identification $N/\Gamma\cap N=N\Gamma/\Gamma$ is given by the correspondence
\begin{equation}\label{corrMostow} 
n^{\prime} \Gamma  \cap N  \in N/\Gamma\cap N    \mapsto n^{\prime} \Gamma\in  N\Gamma/\Gamma. 
\end{equation}
This fiber bundle is called  the \emph{Mostow bundle} of $G/\Gamma$.
Its structure group   is $N\Gamma/\Gamma_{0}$,  where $\Gamma_{0}$ is the largest normal subgroup of $\Gamma$ which is normal in $N\Gamma$,  and the action is by  left translations (see \cite{St}).

For a maximal torus $T$ of ${\mathcal A}(\Ad(G))$, consider the group  $T\ltimes G$.
Take the centralizer $C=C_{T}(G)$.
Then $C$ is a simply connected nilpotent subgroup of $G$ such that $G=C\cdot N$ (see \cite{Dekimpe}).
Note that $N\Gamma=N\cdot (\Gamma\cap C)$.

Suppose that $N$ admits a left-invariant complex structure $J_{N}$.
For $nc\in N\cdot (\Gamma\cap C)$, using the correspondence \eqref{corrMostow}, we have
\[nc\cdot ( n^{\prime} \Gamma\cap N )=ncn^{\prime}c^{-1}  \Gamma\cap N.\]
Hence the structure group  $N\Gamma/\Gamma_{0}$ consists of holomorphic transformations if and only if for any $c\in\Gamma\cap C$ we have $\Ad_{c}\circ J_{N}= J_{N}\circ \Ad_{c}$ on $\frak n$.

%We have:
\begin{prop}\label{holmos}
Suppose that for any $c\in\Gamma\cap C$ we have $\Ad_{c}\circ J_{N}= J_{N}\circ \Ad_{c}$ on $\frak n$.
Then there exist a finite index subgroup $\tilde\Gamma \subset \Gamma$, an  $S$-modification $G^{S}$ containing $\tilde\Gamma$, where $S$ is a complement of a maximal torus  of $\mathcal A (\Ad (\tilde \Gamma))$,  and a simply connected nilpotent subgroup $C^{\prime} \subset G^{S}$  such that $G^{S}=C^{\prime}\cdot N$ and  for any $c^{\prime}\in C^{\prime}$ $\Ad_{c^{\prime}}\circ J_{N}= J_{N}\circ \Ad_{c^{\prime}}$ on $\frak n$.
\end{prop}
\begin{proof}
Suppose that for any $c\in\Gamma\cap C$ we have $\Ad_{c}\circ J_{N}= J_{N}\circ \Ad_{c}$ on $\frak n$.
We can assume that ${\mathcal A}(\Ad(\Gamma))$ is connected, otherwise we can pass to a  finite index subgroup of $\Gamma$.
Let $T$ be a maximal torus of ${\mathcal A}(\Ad(G))$ and $S^{\bot}$ a maximal torus of ${\mathcal A}(\Ad(\Gamma))$.
Consider the semi-direct product $T\ltimes G$.
Then this group is a real algebraic group as $T\ltimes G=T\ltimes U_{G}$, where $U_G$ is the nilshadow of $G$.
Therefore, the  nilradical $N$ of $G$ is a unipotent algebraic subgroup of $U_{G}$.
Take the Zariski-closures ${\mathcal A}( C)$ and ${\mathcal A}(\Gamma\cap C)$ in $T\ltimes U_{G}$.
Since for any $c\in \Gamma\cap C$ we have $\Ad_{c}\circ J_N=J_N \circ \Ad_{c}$ on $\frak n$ and the action of  $T\ltimes U_{G}$ on $\frak n$ is algebraic, for any $c^{\prime}\in {\mathcal A}(\Gamma\cap C)$ we have $\Ad_{c^{\prime}}\circ J_N =J_N \circ \Ad_{c^{\prime}}$.
For the splitting $T\ltimes G=T\ltimes U_{G}$,
we have $C=C_{T\ltimes U_{G}}(T)\cap G$ (see \cite{Dekimpe}).
Hence $C\cap\Gamma=C_{T\ltimes U_{G}}(T)\cap \Gamma$.
Since the Zariski closures of $\Gamma$ and $G$ in $T\ltimes U_{G}$ are $S^{\bot}\ltimes U_{G}$ and $T\ltimes U_{G}$ respectively, ${\mathcal A}( C)$ and ${\mathcal A}(\Gamma\cap C)$ have the same unipotent radical $U^{\prime}=U_{G}\cap C_{T\ltimes U_{G}}(T)$ and we get
${\mathcal A}( C)=T\ltimes U^{\prime}$ and ${\mathcal A}(\Gamma\cap C)=S^{\bot}\ltimes U^{\prime}$.
Consider the modification $G^{S}$.
Then $G^{S}=C^{\prime}\cdot N$ with $C^{\prime}=\{\pi_{S}(c)^{-1}\cdot c\vert c\in C\}$.
Since  $C^{\prime}\subset S^{\bot}\ltimes U^{\prime}={\mathcal A}(\Gamma\cap C)$, for every $c^{\prime}\in C^{\prime}$ we obtain $\Ad_{c^{\prime}}\circ J_N =J_N \circ \Ad_{c^{\prime}}$.
\end{proof}

Suppose $G$ has a left-invariant complex structure $J$.
Then in the above settings, we consider the following condition.
\begin{condition}\label{condd}
 $J(\frak n)=\frak n$ and  $\Ad_{c}\circ J= J\circ \Ad_{c}$ on $\frak n$, for any $c\in C$.
\end{condition}
Note that if the above condition is satisfied, then  the Mostow fibration of $G/\Gamma$ is holomorphic.
If $(G, J)$ is a complex Lie group, then we can take $N$ and $C$ as complex subgroups and hence the condition \ref{condd} holds.

We can state the following 

\begin{cor}\label{ficc}
Let $G/\Gamma$ be a solvmanifold. 
Suppose that $G$ has a  left-invariant complex structure $J$ such that $J(\frak n)=\frak n$.
Then the following two conditions are equivalent.

$(1)$  $\Ad_{c}\circ J_{N} (X)= J_{N}\circ \Ad_{c} (X)$ on $\mathfrak n$,   for every $c\in C$.

$(2)$ $(\Ad_{s})_g\circ J_{N}=J_{N} \circ (\Ad_{s})_g$, for every $g\in G$,  and the Mostow fibration of $G/\Gamma$ is holomorphic.
\end{cor}
\begin{proof}
Since $\Ad_{s}$ is identified with the map $G=C\cdot N\ni cn\mapsto (\Ad_{c})_{s}\in \Aut(\g)$, it  follows that 
 $(1)\Rightarrow (2)$.

Suppose that the condition $(2)$ holds. 
Since $T= {\mathcal A}(\Ad_{S}(G))$, using the splitting $T\ltimes G=T\ltimes U$,
we have $C=C_{T\ltimes U}(T)\cap G$.
Like in the proof of Proposition \ref{holmos}, we have 
\[\pi_{S}(c)^{-1}\circ \Ad_{c}\circ  J_{N}= J_{N}\circ \pi_{S}(c)^{-1}\circ \Ad_{c}.
\]
From the conditions $(\Ad_{s})_g\circ J=J \circ (\Ad_{s})_g$ and $\pi_{S}(c) \in T$, we get that  $\pi_{S}(c)  \circ J_{N}= J_{N}  \circ \pi_{S}(c)$.
Hence we obtain
\[\Ad_{c}\circ  J_{N}=\pi_{S}(c) \circ \pi_{S}(c)^{-1}\circ \Ad_{c}\circ J_{N}= J_{N}  \circ \pi_{S}(c) \circ \pi_{S}(c)^{-1}\circ \Ad_{c}= J_{N}  \circ \Ad_{c}\, ,
\]
and therefore also the assertion $(2)\Rightarrow (1)$ follows.
\end{proof}

Using Proposition~\ref{holmos} it is possible to obtain a left-invariant complex structure on a modification of $G$ starting with (a not necessarily integrable) almost complex structure.
We consider the sum $\g=\frak c+\frak n$ associated with the product $G=C\cdot N,$ where $\frak c$ is the Lie algebra of the subgroup $C\subset G$ as above.
We suppose that $G$ admits a left-invariant almost complex structure $J$ such that $J(\frak c)\subset \frak c$, $J(\frak n)\subset \frak n$ and $J$ is integrable on $\frak c$ and $\frak n$, i.e.,
$N_{J}(X,Y)=0$ for any $X,Y\in \frak c$ and  for any $X,Y \in \frak n,$ where  $N_{J}$ is the Nijenhuis tensor for $J$.
In this case, since it is possible that $N_{J}(X,Y)\not=0$ for $X\in \frak c$ and $Y\in \frak n$,
$J$ may not   be integrable on $G$.
By Proposition~\ref{holmos}, we obtain the following result.
\begin{cor}
Let $G$ be a simply connected solvable Lie group with nilradical $N$ and $J$  a left-invariant almost complex  structure $J$ such  that $J(\frak c)\subset \frak c$, $J(\frak n)\subset \frak n$ and $J$ is integrable on $\frak c$ and $\frak n$.
Suppose that $G$ admits a lattice $\Gamma$ such that,   for any $c\in\Gamma\cap C$,  we have $\Ad_{c}\circ J_{N}= J_{N}\circ \Ad_{c}$ on $\frak n$. 
Then there exist a finite index subgroup $\tilde\Gamma\subset \Gamma$ and a $S$-modification $G^{S}$ containing $\tilde\Gamma$ such that $G^{S}$ admits a left-invariant  complex structure $\tilde J$. 
\end{cor}
\begin{proof}
We take a $S$--modification $G^{S}$ as in  Proposition~\ref{holmos}.
Then we have the sum $\g^{S}=\frak c^{\prime} + \frak n,$ where $\frak c^{\prime}$ is the Lie algebra of $C^{\prime}$.
By the construction of $C^{\prime}$ (see the proof of Proposition~\ref{holmos}),  we have an isomorphism $\frak c^{\prime}\cong \frak c$.
Thus, corresponding to the almost complex structure $J$, we can define the left-invariant complex structure   $\tilde J$ on $G^{S}$ such that 
 $\tilde J(\frak c^{\prime})\subset \frak c^{\prime}$, $\tilde J(\frak n)\subset \frak n$  and
$N_{\tilde J}(X,Y)=0$ for any $X,Y\in \frak c^{\prime}$ and   for any $X,Y \in \frak n$.
As in Proposition~\ref{holmos},  for  every $c^{\prime}\in C^{\prime}$ we have $\Ad_{c^{\prime}}\circ \tilde J_N =\tilde J_N \circ \Ad_{c^{\prime}}$ and hence
for any $X\in \frak c^{\prime}$ and $Y\in \frak n$ we have $[X,\tilde JY]=\tilde J[X,Y]$.
This implies that   $N_{\tilde J}(X,Y)=0$ for  any $X\in \frak c^{\prime}$ and $Y\in \frak n$ and hence $\tilde J$ is integrable.
\end{proof}
This corollary and its proof give the following useful result.

\begin{cor}\label{C-mod}
Let $G$ be a simply connected solvable Lie group with nilradical $N$.  
Suppose that
\begin{enumerate}
\item $G=\C^{n}\ltimes_{\phi} N$ and $G$ admits a lattice $\Gamma=\Gamma_{\C^{n}}\ltimes \Gamma_{N}$;
\item $N$ has a left-invariant complex structure $J_{N}$ such that
% the Mostow fibration \[N/\Gamma_{N}\to G/\Gamma\to \C^{n}/\Gamma_{\C^{n}}\]is holomorphic.\footnote{However we do not assume that the almost complex structure  $J_{\C^{n}}\oplus J_{N}$ on $\g=\C^{n}\ltimes_{\phi} \frak n$ is integrable.}
  for any $t \in \Gamma_{\C^{n}}$ we have $\phi(t)\circ J_{N}=J_{N}\circ \phi(t)$. 
\end{enumerate}
Then there exists a finite index subgroup $\tilde\Gamma\subset \Gamma$ and a $S$-modification $G^{S}=\C^{n}\ltimes_{\tilde\phi} N$ containing $\tilde\Gamma$ such that, for any $t\in \C^{n}$, we have
$\tilde\phi(t)\circ J_{N}=J_{N}\circ \tilde\phi(t)$. 
In particular the almost complex structure  $J_{\C^{n}}\oplus J_{N}$ on $\g^{S}=\C^{n}\ltimes_{\tilde\phi} \frak n$ is integrable.
\end{cor}

\begin{ex}
Let $G=\R^{2}\ltimes_{\phi} \R^{8}$ with
{\small \[\phi(x,y)=\left(\begin{array}{cccccccc}
e^{x}\cos y&-e^{x}\sin y&0&0&0&0&0&0\\
e^{x}\sin y&e^{x}\cos y&0&0&0&0&0&0\\
0&0&e^{x}\cos y&e^{x}\sin y&0&0&0&0 \\
0&0&-e^{x}\sin y&e^{x}\cos y&0&0&0&0\\
0&0&0&0&e^{-x}\cos y&-e^{-x}\sin  y&0&0\\
0&0&0&0&e^{-x}\sin y&e^{-x}\cos  y&0&0\\
0&0&0&0&0&0&e^{-x}\cos y&e^{-x}\sin  y\\
0&0&0&0&0&0&-e^{-x}\sin y&e^{-x}\cos  y
\end{array}\right).
\]}
The Lie algebra $\g$ of $G$ is then $\g=\langle X,Y\rangle \ltimes_\phi \frak n$ where $ \frak n=\langle V_{1},W_{1}, V_{2},W_{2},V_{3},W_{3},V_{4},W_{4}\rangle$ with
\[[X,V_{1}]=V_{1},\;\; [X,W_{1}]=W_{1},\;\;[Y,V_{1}]=W_{1},\;\; [Y,W_{1}]=-V_{1},\]
\[[X,V_{2}]=V_{2},\;\; [X,W_{2}]=W_{2},\;\; [Y,V_{2}]=-W_{2},\;\; [Y,W_{2}]=V_{2},\]
\[[X,V_{3}]=-V_{3},\;\; [X,W_{3}]=-W_{3},\;\; [Y,V_{3}]=W_{3},\;\; [Y,W_{3}]=-V_{3},\]
\[[X,V_{4}]=-V_{4},\;\; [X,W_{4}]=-W_{4},\;\; [Y,V_{4}]=-W_{4},\;\; [Y,W_{4}]=V_{4}.\]
Consider  on $\R^8$ the left-invariant complex structure $J_{\R^{8}}$ 
given by 
\[J_{\R^{8}}V_{1}=V_{2},\;\; J_{\R^{8}}W_{1}=W_{2},\;\;J_{\R^{8}}V_{3}=V_{4},\;\;J_{\R^{8}}W_{3}=W_{4}.\]
Then, since $[Y,J_{\R^8} V_{1}]=-W_{2}$ and $J_{\R^8}[Y, V_{1}]=W_{2}$, for $y\not=0$ we have
$\phi(0,y)\circ  J_{\R^8} \not= J_{\R^8}\circ \phi(0,y)$  and for $J_{\R^{2}}$ with $J_{\R^{2}}(X)=Y$ the almost complex structure $J_{\R^{2}}\oplus  J_{\R^{8}}$ is not integrable. 
We can take as a lattice of $G$ the subgroup  $\Gamma=(a\Z+2\pi\sqrt{-1}\Z)\ltimes \Gamma^{\prime\prime}$ for some $a\in \R$ and some lattice $\Gamma^{\prime\prime}$ in $\R^{8}$.
Then we can apply Corollary \ref{C-mod}. Indeed,  $J_{\R^8}$ satisfies the  condition
 $\phi(t)\circ J_{\R^8}=J_{\R^8}\circ \phi(t)$,  for every $t \in  a\Z+2\pi\sqrt{-1}\Z$. The modification $G^S$ is given by 
$G^{S}=\R^{2}\ltimes_{\tilde\phi} \R^{8}$ with
\[\tilde \phi(x,y)=\left(\begin{array}{cccccccc}
e^{x}&0&0&0&0&0&0&0\\
0&e^{x}&0&0&0&0&0&0\\
0&0&e^{x}&0&0&0&0&0 \\
0&0&0&e^{x}&0&0&0&0\\
0&0&0&0&e^{-x}&0&0&0\\
0&0&0&0&0&e^{-x}&0&0\\
0&0&0&0&0&0&e^{-x}&0\\
0&0&0&0&0&0&0&e^{-x}
\end{array}\right).\]
Moreover, $G^{S}$ contains $\Gamma$ and 
 $\tilde\phi(x,y)\circ J_{\R^8}= J_{\R^8}\circ \tilde\phi(x,y),$ for every $(x,y)\in \R^{2}$. Thus  $J_{\R^{2}}\oplus  J_{\R^{8}}$ is integrable on $\cg^S$.
\end{ex}

\section{Dolbeault cohomology}\label{Dolb_gen}

The aim of this section is to construct  a differential bigraded algebra (DBA) which allows to compute the Dolbeault cohomology  of a solvmanifold   $(G/\Gamma, J)$ endowed with an invariant complex structure $J$, in the case that  the  Mostow fibration is holomorphic  and $J$ commutes with  $\Ad_s$.

We first recall some results in \cite{Kas1}.  Consider a solvable Lie group $G$ which is a semi-direct product of the form $\C^{n}\ltimes _{\phi}N$ where:
\begin{itemize}
\item  $N$ is a simply connected nilpotent Lie group with the Lie algebra $\frak n$ and a left-invariant complex structure $J_N$;
\item for any $t\in \C^{n}$, $\phi(t)$ is a holomorphic automorphism of $(N,J_N)$;
\item $\phi$ induces a semi-simple action on the Lie algebra $\frak n$ of $N$.
\item $G$ has a lattice $\Gamma$ (then $\Gamma$ can be written by $\Gamma=\Gamma^{\prime}\ltimes_{\phi}\Gamma^{\prime\prime}$ such that $\Gamma^{\prime}$ and $\Gamma^{\prime\prime}$ are  lattices of $\C^{n}$ and $N$ respectively and for any $t\in \Gamma^{\prime}$ the action $\phi(t)$ preserves $\Gamma^{\prime\prime}$); 
\item the inclusion $\bigwedge^{\ast,\ast}\frak n^{\ast}\subset A^{\ast,\ast}(N/\Gamma^{\prime\prime})$ induces an isomorphism 
\[H^{\ast,\ast}_{\bar\partial}(\frak n)\cong H^{\ast,\ast}_{\bar\partial }(N/\Gamma^{\prime\prime})\, .\]
\end{itemize}
By using the Borel spectral sequence (see \cite{FW})  of the holomorphic fibration 
\[N/\Gamma_{N}\to G/\Gamma\to \C^{n}/\Gamma_{\C^{n}}\]
determined by the splitting (in fact this spectral sequence is degenerate at $E_{2}$ (see \cite[Section. 4]{Kas1}),
the third author  constructed in  \cite{Kas1}  an explicit finite dimensional sub-DBA of the Dolbeault complex  $A^{\ast,\ast}(G/\Gamma)$ which computes the Dolbeault cohomology of $(G/\Gamma, J)$.

\smallskip

Given a complex representation $V_{\rho}$ of a simply connected complex abelian Lie group we define as in \cite{Raghunathan} the holomorphic flat bundle $E_{\rho} = (A \times  V_{\rho})/\Gamma$ given by the equivalence relation $(\gamma g, \rho(\gamma) v) \sim (g, v)$ for $g \in A$, $v \in V_{\rho}$ and $\gamma \in \Gamma$. We can prove the following

\begin{lemma}\label{tori}
Let $A$ be a simply connected complex abelian Lie group with a lattice $\Gamma$, $\frak a$ the Lie algebra of $A$ and $\rho: A\to GL(V_{\rho})$  a representation of $A$ on a complex vector space $V_{\rho}$.
Consider the generalized weight decomposition (i.e. the weight decomposition of the semisimple part $\rho_{s}$ of $\rho$)
\[V_{\rho}=\bigoplus V_{\rho_{\alpha_{i}}}.
\]
Take unitary characters $\beta_{i}$ such that $\alpha_i\beta_i^{-1}$ are holomorphic  on $A$ as in \cite[Lemma 2.2]{Kas1}.

Then we have an isomorphism
\[H^{\ast,\ast}(A/\Gamma, E_{\rho})\cong \bigoplus_{\beta_{i}\vert\Gamma=1} H^{\ast,\ast}(\frak a, V_{\beta^{-1}_{i}\rho_{\alpha_{i}}})\, ,
\]
where  $E_{\rho}$ is the holomorphic flat vector bundle induced by the representation $\rho$.
\end{lemma}

\begin{proof}
Let $E_{\rho_{\alpha_{i}}}$ be the holomorphic flat vector bundle induced by the representation $\rho_{\alpha_{i}}$.
Since $\rho_{\alpha_{i}}$ is triangularizable, we have a sequence
\[E_{\rho_{\alpha_{i}}}=E^{0}_{\rho_{\alpha_{i}}}\supset E^{1}_{\rho_{\alpha_{i}}}\supset E^{2}_{\rho_{\alpha_{i}}}\dots \supset E^{j}_{\rho_{\alpha_{i}}}=(A/\Gamma)\times \{0\}
\]
 of sub-bundles such that $E^{k}_{\rho_{\alpha_{i}}}/ E^{k+1}_{\rho_{\alpha_{i}}}\cong L_{\alpha_{i}}$ where $L_{\alpha_{i}}$ is the holomorphic flat line bundle defined by the character $\alpha_{i}$.
It is known that 
\[H^{\ast,\ast}(A/\Gamma, L_{\alpha_{i}})=0
\] 
if $\beta_{i}\vert_{\Gamma}\not=1$ see \cite[Lemma 2.3, Proposition 2.4]{Kas1}.
Hence, inductively, if  $\beta_{i}\vert_{\Gamma}\not=1$,  we get 
\[H^{\ast,\ast}(A/\Gamma, E_{\rho_{\alpha_{i}}})=0.
\]
If $\beta_{i}\vert_{\Gamma}=1$, we have
\[H^{\ast,\ast}(A/\Gamma, E_{\rho_{\alpha_{i}}})\cong H^{\ast,\ast}(A/\Gamma, E_{\beta_i^{-1} \rho_{\alpha_{i}}}).\]
Consider the natural inclusion
\[\bigwedge ^{\ast,\ast}{\frak a}^{\ast}\otimes  V_{\beta^{-1}_{i}\rho_{\alpha_{i}}}\subset A^{\ast,\ast}(A/\Gamma, E_{\beta^{-1}_{i}\rho_{\alpha_{i}}}).
\]
Then since the line bundle $L_{\beta_{i}^{-1}\alpha_{i}}$ is trivial,  inductively  this inclusion   induces a cohomology isomorphism.
Hence the lemma follows.
\end{proof}

Let $G$ be a simply connected solvable Lie group  with a left-invariant complex structure $J$,  $\Gamma$  be a   lattice of $G$ and $N$ be the nilradical of $G$.

In the same setting as in Section \ref{hol_Mostow}, we suppose that  Condition \ref{condd} holds.
By Corollary \ref{ficc}, we have $(\Ad_{s})_g\circ J=J\circ (\Ad_{s})_g$ and we can take  a basis $X_{1},\dots, X_{n},Y_{1},\dots, Y_{m}$ of $\g^{1,0}$ such that $Y_{1},\dots, Y_{m}$ is a basis of $\frak n^{1,0}$ and  $(\Ad_{s})_g Y_{i}=\alpha_{i}(g)Y_{i}$ for any $g\in G$.
Consider unitary characters $\beta_{i}$ and $\gamma_{i}$ such that $\alpha_{i}\beta_{i}^{-1}$ and $\bar\alpha_{i}\gamma^{-1}$ are holomorphic on $G/N$ as in \cite[Lemma 2.2]{Kas1}.
Let $x_{1},\dots,x_{n},y_{1}\dots,y_{m}$ be the dual basis of $X_{1},\dots, X_{n},Y_{1},\dots, Y_{m}$.
Define the differential bigraded algebra (DBA)
\[B_{\Gamma}^{p,q}=\bigoplus_{a+c=p,b+d=q}\bigwedge^{a} \langle x_{1},\dots,x_{n}\rangle \otimes \bigwedge^{b} \langle \bar x_{1},\dots,\bar x_{n}\rangle \otimes {\rm span} \left\langle\beta_{I}\gamma_{J} y_{I}\wedge\bar y_{J}{\Big \vert} \begin{array}{cc}\vert I\vert=c,\, \vert J\vert=d \\  (\beta_{J}\gamma_{L})_{\vert_{\Gamma}}=1\end{array}\right\rangle
\]

\begin{theorem}\label{dolbb}

Suppose that the complex nilmanifold $(N/N\cap \Gamma,J_{N})$ admits good Dolbeault cohomology (i.e.  $H_{\overline{\partial}}^{*,*} (N/\Gamma 	\cap N) \cong H_{\overline{\partial}}^{*,*} (\frak n)$).
Then the inclusion $B_{\Gamma}^{\ast,\ast}\subset A^{\ast,\ast}(G/\Gamma) $ induces a cohomology isomorphism.
\end{theorem}
\begin{proof}
By 
\[A^{\ast,\ast}(G/\Gamma)=C^{\infty}(G/\Gamma)\otimes \bigwedge ^{\ast,\ast}\g^{\ast}\]
and
\[\bigwedge \langle x_{1},\dots,x_{n},y_{1},\dots, y_{m}\rangle \otimes \bigwedge \langle \bar x_{1},\dots,\bar x_{n},\bar y_{1},\dots, \bar y_{m}\rangle=\bigwedge ^{\ast,\ast} \cg^*,\]
we define the filtration 
{\small{
$$
%\begin{multline*}
F^{r} A^{p,q}(G/\Gamma)
%\\
=\bigoplus_{\begin{subarray}{l}
        a+c=p, \\b+d=q,\\ a+b\ge r
      \end{subarray}}
C^{\infty}(G/\Gamma)\otimes \bigwedge ^{a}\langle x_{1},\dots,x_{n}\rangle \otimes \bigwedge^{b} \langle \bar x_{1},\dots,\bar x_{n}\rangle\otimes \bigwedge ^{c}\langle y_{1},\dots, y_{m}\rangle\otimes \bigwedge ^{d}\langle \bar y_{1},\dots, \bar y_{m}\rangle.
$$
%\end{multline*}
}}
This filtration yields the Borel spectral sequence $\,^{\ast} E^{\ast,\ast}_{\ast}$ of the holomorphic Mostow fibration (see \cite{FW}).
We have 
\[\,^{p}E_{1}^{s,t}=\bigoplus_{i}A^{i,s-i}(G/\Gamma, {\bf H}^{p-i, t+i}(N/N\cap \Gamma))
\]
and
\[\,^{p}E_{2}^{s,t}=\bigoplus_{i}H^{i,s-i}(G/\Gamma, {\bf H}^{p-i, t+i}(N/N\cap \Gamma))
\]
where ${\bf H}^{\ast, \ast}(N/N\cap \Gamma)$ is the holomorphic vector bundle $\bigcup_{b \in G/N\Gamma}H^{\ast,\ast}_{\bar\partial}(p^{-1}(b))$ over $G/N\Gamma$ for the fibration $p:G/\Gamma\to G/N\Gamma$.
By the filtration 
associated to $B_{\Gamma}^{p,q}$ 
 we get
\[F^{r}B_{\Gamma}^{p,q}=\bigoplus_{\begin{subarray}{l}
        a+c=p, \\b+d=q,\\ a+b\ge r
      \end{subarray}}\bigwedge^{a} \langle x_{1},\dots,x_{n}\rangle \otimes \bigwedge^{b} \langle \bar x_{1},\dots,\bar x_{n}\rangle \otimes {\rm span} \left\langle\beta_{I}\gamma_{J} y_{I}\wedge\bar y_{J}{\Big \vert} \begin{array}{cc}\vert I\vert=c,\, \vert J\vert=d \\  (\beta_{J}\gamma_{L})_{\vert_{\Gamma}}=1\end{array}\right\rangle
\]
and $\bar\partial F^{r}B_{\Gamma}^{p,q}\subset F^{r}B_{\Gamma}^{p,q+1}$.
Hence the filtration induces the spectral sequence $\,^{\ast}\tilde E_{\ast}^{\ast,\ast}$ of $B_{\Gamma}^{\ast,\ast}$ and the homomorphism
\[\hat i: \tilde E_{\ast}^{\ast,\ast}\to  E_{\ast}^{\ast,\ast}\]
determined by the inclusion $i:B^{\ast,\ast}_{\Gamma}\to  A^{\ast,\ast}(G/\Gamma)$.
We will prove that this homomorphism is an isomorphism at $E_{2}$-term.

We first show  that this homomorphism is injective.
Since $G$ is unimodular, we can take a bi-invariant Haar measure $\eta$ such that $\int_{G/\Gamma}\eta=1$.
Consider the map $\phi:A^{\ast,\ast}(G/\Gamma)\to B_{\Gamma}^{\ast,\ast}$ defined by
\[\phi \left (\sum f_{IJKL}x_{I}\wedge y_{J}\wedge \bar x_{K}\wedge \bar y_{L} \right )=\sum_{(\beta_{I}\gamma_{J} )_{\vert_{\Gamma}}=1}  \left(\int_{G/\Gamma}\frac{f_{IJKL}}{\beta_{I}\gamma_{J} }\eta\right) \beta_{I}\gamma_{J}  x_{I}\wedge y_{J}\wedge \bar x_{K}\wedge \bar y_{L}\]
of the terms $\sum f_{IJKL}x_{I}\wedge y_{J}\wedge \bar x_{K}\wedge \bar y_{L}$.
It is known that for any $\mathcal{C}^\infty$ function $F$ on $G/\Gamma$ and any left-invariant vector field $A$,
we have $\int_{G/\Gamma} A(F)\eta=0$ (see \cite{Bel}).
Hence we get
\[
\begin{array}{ll}
%\begin{multline*}
\phi\left(\bar\partial(\sum f_{IJKL}x_{I}\wedge y_{J}\wedge \bar x_{K}\wedge \bar y_{L})\right)
%\\
&=\sum_{(\beta_{I}\gamma_{J} )_{\vert_{\Gamma}}=1}  \left(\int_{G/\Gamma}\frac{f_{IJKL}}{\beta_{I}\gamma_{J} }\eta\right) \bar\partial(\beta_{I}\gamma_{J}  x_{I}\wedge y_{J}\wedge \bar x_{K}\wedge \bar y_{L})\\[10pt]
&=\bar\partial(\phi(\sum f_{IJKL}x_{I}\wedge y_{J}\wedge \bar x_{K}\wedge \bar y_{L}))
\end{array}
\]
%\end{multline*}
and consequently $\phi$ is a homomorphism of cochain complexes.
By the definition of $\phi$, we have $\phi(F^{r}A^{p,q}(G/\Gamma))\subset F^{r}B^{p,q}_{\Gamma}$ and $\phi\circ i=id_{B^{\ast,\ast}_{\Gamma}}$ for the inclusion $i:B^{\ast,\ast}_{\Gamma}\to A^{\ast,\ast}(G/\Gamma)$.
Hence $\phi$ induces a homomorphism $\hat\phi: E_{\ast}^{\ast,\ast}\to\tilde E_{\ast}^{\ast,\ast}$ such that $\hat\phi \circ \hat i=id_{E_{\ast}^{\ast,\ast}}$
and for any term, $\hat i: \tilde E_{\ast}^{\ast,\ast}\to  E_{\ast}^{\ast,\ast}$ is injective.

Consider the action $\rho:G/N\to \Aut (H^{\ast,\ast}(\frak n))$ induced by the adjoint action of $G$ on $\cn$.
Then since $H^{\ast,\ast}(N/N\cap \Gamma)\cong H^{\ast,\ast}(\frak n)$, we get that ${\bf H}^{\ast,\ast}(N/N\cap \Gamma)$ is the holomorphic flat bundle induced by $\rho$.
Take the generalized weight decomposition
\[H^{\ast,\ast}(\frak n)=\bigoplus V_{\rho_{\delta_{i}}}\]
and take the unitary characters $\epsilon_{i}$ such that $\delta_{i}\epsilon_{i}^{-1}$ are holomorphic.
By Lemma \ref{tori}, we have
\[\bigoplus \,^{\ast}E_{2}^{\ast,\ast}=\bigoplus H^{\ast,\ast}(G/N\Gamma, {\bf H}^{\ast, \ast}(N/N \cap \Gamma))\cong  H^{\ast,\ast}\left(\g/\frak n, \bigoplus_{\epsilon_{i}\vert_{\Gamma}=1} V_{\epsilon^{-1}_{i}\rho_{\delta_{i}}}\right ).
\]
Observe that $\rho=\Ad_s$ on $G/N \Gamma$. 
Let $\sigma: G/N\to \Aut(\bigwedge^{\ast,\ast}
\frak n)$ be the action induced by $\Ad_{s}$ and  \[\bigwedge^{\ast,\ast}
\frak n=\bigoplus V_{\sigma_{\delta_{i}}}
\]
the generalized weight decomposition. 
Then the $V_{\sigma_{\delta_{i}}}$ are sub-cochain complexes of $\bigwedge^{\ast,\ast}
\frak n$ and this decomposition is a direct sum of cochain complexes.
Hence
\[B_{\Gamma}^{\ast,\ast}=\bigwedge^{\ast} \langle x_{1},\dots,x_{n}\rangle \otimes \bigwedge^{\ast} \langle \bar x_{1},\dots,\bar x_{n}\rangle \otimes \bigoplus_{\epsilon_{i}\vert_{\Gamma}=1} \epsilon_{i}V_{\sigma_{\delta_{i}}}.
\]
Since $\sigma: G/N\to \Aut(\bigwedge^{\ast,\ast}
\frak n)$ induces the semisimple part of the action $\rho:G/N\to \Aut (H^{\ast,\ast}(\frak n))$, 
we have
\[H^{\ast,\ast}(V_{\sigma_{\delta_{i}}})\cong V_{\rho_{\delta_{i}}}\, ,\]
and consequently 
\[\bigoplus \,^\ast \tilde E^{\ast,\ast}_{1}=\bigoplus \bigwedge^{\ast} \langle x_{1},\dots,x_{n}\rangle \otimes \bigwedge^{\ast} \langle \bar x_{1},\dots,\bar x_{n}\rangle \otimes  \epsilon_{i}V_{\rho_{\delta_{i}}}\cong \bigoplus_{\epsilon_{i}\vert_{\Gamma}=1}\bigwedge^{\ast,\ast} (\g/\frak n)\otimes  V_{\epsilon_{i}^{-1}\rho_{\delta_{i}}}
\]
and 
\[\bigoplus \,^\ast \tilde E^{\ast,\ast}_{2}\cong H^{\ast,\ast}\left (\g/\frak n, \bigoplus_{\epsilon_{i}\vert_{\Gamma}=1} V_{\epsilon^{-1}_{i}\rho_{\delta_{i}}} \right ).\]
So  the theorem follows.
\end{proof}

\begin{ex}\label{nonsplit}
Let $\g={\rm {span}}\langle A_{1}, A_{2}, B_{1}, B_{2}, C_{1}, C_{2}, C_{3}, C_{4}, D_{1}, D_{2},D_{3},D_{4}\rangle$ be the Lie algebra with structure equations
\[[A_{1},A_{2}]=B_{1},
\]
\[[A_{1},C_{1}]=k_{0}C_{1},\;[A_{1},C_{2}]=k_{0}C_{2},\; [A_{1},C_{3}]=-k_{0}C_{3},\;  [A_{1},C_{4}]=-k_{0}C_{4},\]
\[[A_{2},D_{1}]=k_{0}D_{1},\;[A_{2},D_{2}]=k_{0}D_{2},\; [A_{2},D_{3}]=-k_{0}D_{3},\;  [A_{1},D_{4}]=-k_{0}D_{4},\]
where $k_{0}$ is a real constant number so that
the matrix $\left(  \begin{array}{cc}
e^{k_{0}}&0\\
0&e^{-k_{0}}
\end{array}\right)$ is conjugate to an element of $SL(2,\Z)$.
Then, the nilradical has Lie algebra $\frak n=\langle B_{1}, B_{2}, C_{1}, C_{2}, C_{3}, C_{4}, D_{1}, D_{2},D_{3},D_{4}\rangle$ and the extension
\[0\to \frak n\to \g\to \g/\frak n\to 0
\]
cannot split.
Consider the complex structure $J$ on $\g$ defined by 
\[JA_{1}=A_{2},\; JB_{1}=B_{2},\; JC_{1}=C_{2},\; JC_{3}=C_{4},\; JD_{1}=D_{2},\; JD_{3}=D_{4}.
\]
Let $G$ be the simply connected Lie group  with the Lie algebra $\g$ and $N$ the nilradical.
Then $G$ cannot split as $G=\C\ltimes N$, but on the other hand we have
\[G=(H_{3}(\R)\times \R)\ltimes_{\phi}\C^{4}\, ,\]
where 
\[\phi\left(\left(  \begin{array}{ccc}
1&s_{1}&t_{1}\\
0&1&s_{2}\\
0&0&1
\end{array}\right), t_{2}\right)=\left(  \begin{array}{cccc}
e^{k_{0}s_{1}}&0&0&0\\
0&e^{-k_{0}s_{1}}&0&0\\
0&0&e^{k_{0}s_{2}}&0\\
0&0&0&e^{-k_{0}s_{2}}
\end{array}\right)
\]
and $H_{3}(\R)$ is the three dimensional Heisenberg group.
$G$ has a lattice $\Gamma=(H_{3}(\Z)\times \Z)\ltimes_{\phi}\Gamma_{\C^{4}}$.
We may take $V= {\rm {span}}\langle A_{1},\; A_{2}\rangle$ such that $\cg=V \oplus \cn$ for which $\ad(A)_s (B)=0$ for any $A, B \in V$. The nilpotent subgroup $C=H_{3}(\R)\times \R$ is a nilpotent complement of $N$ such that $G=C \cdot N$.
Then by  Corollary \ref{ficc}, $(G/\Gamma,J)$ satisfies the assumption in Theorem \ref{dolbb}.
Let $X_{1}=A_{1}-\sqrt{-1}A_{2}$, $Y_{1}=B_{1}-\sqrt{-1}B_{2}$, $Y_{2}=C_{1}-\sqrt{-1}C_{2}$, $Y_{3}=C_{3}-\sqrt{-1}C_{4}$, $Y_{4}=D_{1}-\sqrt{-1}D_{2}$ and $Y_{5}=D_{3}-\sqrt{-1}D_{4}$.
Then we get  $\g^{1,0}={\rm {span}}\langle X_{1},\; Y_{1},\;Y_{2},\;Y_{3},\; Y_{4},\; Y_{5}\rangle$ and $\frak n^{1,0}={\rm {span}}\langle Y_{1},\;Y_{2},\;Y_{3},\; Y_{4},\; Y_{5}\rangle$.
Take the dual basis $x_{1},y_{1},y_{2},y_{3},y_{4},y_{5}$ of $X_{1}, \; Y_{1},\;Y_{2},\;Y_{3},\; Y_{4},\; Y_{5}$. 
We have
\[B^{\ast,\ast}_{\Gamma}=\bigwedge \langle x_{1},\, \bar x_{1}\rangle\otimes \bigwedge \langle y_{1},\, \bar y_{1}, \, y_{2}\wedge y_{3},\,  y_{2}\wedge\bar y_{3},\, y_{3}\wedge\bar y_{2},\,  \bar y_{2}\wedge \bar y_{3}, y_{4}\wedge y_{5},\, y_{4}\wedge\bar y_{5},\, y_{5}\wedge\bar y_{4},\,  \bar y_{4}\wedge \bar y_{5} \rangle
\]
and consequently 
{\small{
\[
%\begin{multline*}
H^{\ast,\ast}(G/\Gamma)\\
\cong H^{\ast,\ast}((H_{3}(\R)\times \R)/_{(H_{3}(\Z)\times \Z)})
\otimes \bigwedge \langle  \, y_{2}\wedge y_{3},\,  y_{2}\wedge\bar y_{3},\, y_{3}\wedge\bar y_{2},\,  \bar y_{2}\wedge \bar y_{3}, y_{4}\wedge y_{5},\, y_{4}\wedge\bar y_{5},\, y_{5}\wedge\bar y_{4},\,  \bar y_{4}\wedge \bar y_{5} \rangle.
%\end{multline*}
\]
}}
\end{ex}

\begin{rem}
Suppose that $G$  is the semi-direct product $\C^{n}\ltimes _{\phi} L$  such that 
$L$  is a simply connected nilpotent Lie group with the Lie algebra $\frak l$.
Then in general $L$ is not necessarily the nilradical of $G$.
On the other hand, in this case, we have $\Ad_{s}=\id_{\C^{n}}\oplus\phi_{s}$ on $\g=\C^{n}\ltimes  \frak l$ where $\phi_{s}$  is a semi-simple part of $\phi:\C^{n}\to \Aut( \frak l)$.
By a similar proof as the one of Theorem \ref{dolbb},  as a generalization of the result in \cite{Kas1} we obtain the following theorem.
\begin{theorem}\label{split}
Let $G$ be a solvable Lie group which is the semi-direct product  $\C^{n}\ltimes _{\phi} L$ where:
\begin{itemize}
\item  $L$  is a simply connected nilpotent Lie group with the Lie algebra $\frak l$  and a left-invariant complex structure  $J_L$.
\item For any $t\in \C^{n}$, $\phi(t)$ is a holomorphic automorphism of  $(L,J_L)$.\footnote{Note we do not need to suppose $\phi$ is semi-simple.}
\item $G$ has a lattice $\Gamma$. (Then $\Gamma$ can be written by $\Gamma=\Gamma^{\prime}\ltimes_{\phi}\Gamma^{\prime\prime}$ such that $\Gamma^{\prime}$ and $\Gamma^{\prime\prime}$ are  lattices of $\C^{n}$ and  $L$ respectively and for any $t\in \Gamma^{\prime}$ the action $\phi(t)$ preserves $\Gamma^{\prime\prime}$). 
\item The inclusion  $\bigwedge^{\ast,\ast}\frak l^{\ast}\subset A^{\ast,\ast}(H/\Gamma^{\prime\prime})$ induces an isomorphism
\[H^{\ast,\ast}_{\bar\partial}(\frak l)\cong H^{\ast,\ast}_{\bar\partial }(L/\Gamma^{\prime\prime}, J_L).\]
\end{itemize}
Consider a basis $X_{1},\dots, X_{n},Y_{1},\dots, Y_{m}$ of $\g^{1,0}$ such that $Y_{1},\dots, Y_{m}$ is a basis of  $\frak l^{1,0}$ and  $\phi_{s}(g)Y_{i}=\alpha_{i}(g)Y_{i}$ for any $g\in \C^{n}$.
Take unitary characters $\beta_{i}$ and $\gamma_{i}$ such that $\alpha_{i}\beta_{i}^{-1}$ and $\bar\alpha_{i}\gamma^{-1}$ are holomorphic on $\C^{n}$ as in \cite[Lemma 2.2]{Kas1}.
Let $x_{1},\dots,x_{n},y_{1}\dots,y_{m}$ be the dual basis of $X_{1},\dots, X_{n},Y_{1},\dots, Y_{m}$ and define the differential bigraded algebra (DBA)
\[B_{\Gamma}^{p,q}=\bigoplus_{a+c=p,b+d=q}\bigwedge^{a} \langle x_{1},\dots,x_{n}\rangle \otimes \bigwedge^{b} \langle \bar x_{1},\dots,\bar x_{n}\rangle \otimes {\rm span} \left\langle\beta_{I}\gamma_{J} y_{I}\wedge\bar y_{J}{\Big \vert} \begin{array}{cc}\vert I\vert=c,\, \vert J\vert=d \\  (\beta_{J}\gamma_{L})_{\vert_{\Gamma}}=1\end{array}\right\rangle
\]
Then the inclusion $B_{\Gamma}^{\ast,\ast}\subset A^{\ast,\ast}(G/\Gamma) $ induces a cohomology isomorphism.
\end{theorem}
\end{rem}

\section{Dolbeault cohomology and modifications}\label{Dolb_mod}

Using the DBA described in the previous Section, we will show that for some cases the Dolbeault cohomology of $(G /\Gamma, J)$ is isomorphic to the invariant Dolbeault cohomology of a modification of the Lie algebra of $G$.

Consider the same assumptions as in Theorem \ref{dolbb} and the DBA
\[B_{\Gamma}^{p,q}=\bigoplus_{a+c=p,b+d=q}\bigwedge^{a} \langle x_{1},\dots,x_{n}\rangle \otimes \bigwedge^{b} \langle \bar x_{1},\dots,\bar x_{n}\rangle \otimes {\rm span}\left\langle\beta_{I}\gamma_{J} y_{I}\wedge\bar y_{J}{\Big \vert} \begin{array}{cc}\vert I\vert=c,\, \vert J\vert=d \\  (\beta_{J}\gamma_{L})_{\vert_{\Gamma}}=1\end{array}\right\rangle
\]
such that  the inclusion $B_{\Gamma}^{\ast,\ast}\subset A^{\ast,\ast}(G/\Gamma) $ induces a cohomology isomorphism.

Let
$$
\rho: G \rightarrow GL(\frak g_{\C})
$$
be a representation such that $\rho(g)X_{i}=X_{i}$, $\rho(g)\bar X_{i}=\bar X_{i}$,  $\rho(g)Y_{i}=\beta_{i} Y_{i}$ and $\rho(g)\bar Y_{i}=\gamma_{i}\bar Y_{i}$.
\begin{lemma}\label{Au}
We denote by $\Aut^{d}_{\frak n_{\C}}(\frak g_{\C})$ the group of automorphisms  of the complex Lie algebra $\frak g_{\C}$ which are diagonalized by the basis $X_{1},\dots,X_{n},\bar X_{1},\dots, \bar X_{n}, Y_{1},\dots ,Y_{m},\bar Y_{1},\dots,\bar Y_{m}$ such that $f\in \Aut^{d}_{\frak n_{\C}}(\frak g_{\C})$ satisfies $fX_{i}=X_{i}$ and $f\bar X_{i}=\bar X_{i}$.
Then for any $g\in G$, we have $\rho(g)\in \Aut^{d}_{\frak n_{\C}}(\frak g_{\C})$.
\end{lemma}
\begin{proof}
For $f\in \Aut^{d}_{\frak n_{\C}}(\frak g_{\C})$, we have
$f[Y_{i},Y_{j}]=[f Y_{i},f Y_{j}]$, $f[Y_{i},\bar Y_{j}]=[f Y_{i},f \bar Y_{j}]$ and $f[\bar Y_{i},\bar Y_{j}]=[f \bar Y_{i},f \bar Y_{j}]$.
Choosing the parameters  $(t_{1},\dots ,t_{m}, s_{1},\dots, s_{m})$ such that $f Y_{i}=t_{i} Y_{i}$ and $f \bar Y_{i}= s_{i} Y_{i}$,
the algebraic group $\Aut^{d}(\frak n_{\C})$ is then defined by  equations of the form
\[t_{[m]}^{I}s_{[m]}^{J}=t_{[m]}^{K}s_{[m]}^{L}
\]
where for a multi-index $I=(i_{1},\dots,i_{m})$ we write $t_{[m]}^{I}=t_{1}^{i_{1}}\cdots t_{m}^{i_{m}}$.
If   the conditions
\[\alpha_{[m]}^{I} \bar\alpha_{[m]}^{J}=\alpha_{[m]}^{K} \bar\alpha_{[m]}^{L}\]
hold,  by  definition of $\beta_{i}$ and $\gamma_{i}$ we get
\[\beta_{[m]}^{I} \gamma_{[m]}^{J}=\beta_{[m]}^{K} \gamma_{[m]}^{L}.\]
Hence, since by construction $(\Ad_{s})_g\in \Aut^{d}_{\frak n_{\C}}(\frak g_{\C})$, we have $\rho(g)\in \Aut^{d}_{\frak n_{\C}}(\frak g_{\C})$.
\end{proof}
We have that $T = {\mathcal A} (\rho (G))$ and thus by the previous lemma $T\subset  \Aut^{d}_{\frak n_{\C}}(\frak g_{\C})$.
 Let $S^{\perp}$ the connected subtorus of $T$ defined by
$$
S^{\perp} =  {\mathcal A} (\rho (\Gamma)).
$$
We assume $S^{\perp}$ is connected.
In case $S^{\perp}$ is not connected, we can take a finite index subgroup $\tilde\Gamma\subset \Gamma$ such that ${\mathcal A} (\rho (\tilde\Gamma))$ is connected.
We have the direct decomposition $T = S \times S^{\perp}$,  where  $S$ is a complement of $S^{\perp}$ in $T$. Consider the projection 
$$
\pi: G \rightarrow T \rightarrow S,
$$
then 
$$
\pi: G \rightarrow (\tilde \beta_1, \ldots, \tilde \beta_m, \tilde \gamma_1, \ldots, \tilde \gamma_m).
$$
Consider the vector spaces
$$
\begin{array}{l}
\tilde {\frak g}^{1,0} =  {\rm {span}}  \langle X_1, \ldots X_n,  \tilde \beta_1^{-1} Y_1, \ldots,  \tilde \beta_m^{-1} Y_m  \rangle,\\[4 pt]
\tilde {\frak g}^{0,1}  = {\rm {span}}  \langle \overline X_1, \ldots  \overline X_n,  \tilde \gamma_1^{-1} \overline  Y_1, \ldots,   \tilde \gamma_m^{-1}   \overline Y_m \rangle,
\end{array}
$$
%define a new  real solvable Lie algebra $\tilde {\frak g}$  (which is not complex)   
%endowed with a complex structure $\tilde J$ which has $\tilde {\frak g}^{1,0}$ has $(1,0)$-vectors.  This Lie algebra $( \tilde {\frak g}, \tilde J)$ can be viewed as a modification of $(\frak g, J)$ and one has the following
with duals
$$
\begin{array}{l}
(\tilde {\frak g}^{1,0})^{\ast} =  {\rm {span}} <x_1, \ldots x_n,   \tilde \beta_1 y_1, \ldots,   \tilde \beta_m y_m >,\\[4 pt]
(\tilde {\frak g}^{0,1})^{\ast}  = {\rm {span}} <\overline x_1, \ldots  \overline x_n,    \tilde \gamma_1\overline y_1, \ldots,    \tilde \gamma_m\overline y_m >,
\end{array}
$$
and the subcomplex $\bigwedge (\tilde {\frak g}^{1,0})^{\ast}\otimes (\tilde {\frak g}^{0,1})^{\ast} $ of the Dolbeault complex $A^{\ast,\ast}(G/\Gamma)$.
Then we can show the following:
\begin{lemma}
$B^{\ast,\ast}_{\Gamma}\subset\bigwedge (\tilde {\frak g}^{1,0})^{\ast}\otimes (\tilde {\frak g}^{0,1})^{\ast} $.
\end{lemma}
\begin{proof}
We regard $T\subset \Aut^{d}(\frak n_{\C})$ as a subset of  the group of  the complex diagonal matrices  $D_{2m}(\C)$.
For 
\[t_{i}\in  {\rm {Char}}(D_{2m}(\C))=\{ t_{1}^{i_{1}}\cdots t_{n}^{i_{n}} s_{1}^{j_{1}}\dots s_{m}^{j_{m}}\vert\; \diag(t_{1},\dots, t_{n},s_{1}\dots s_{m})\in D_{2m}(\C)\},
\]
we denote $f_{i}=(t_{i})_{\vert T}$, $f^{\prime}_{i}=(s_{i})_{\vert T}$ , $g_{i}= (t_{i})_{\vert S}\circ p_{S}$, $g^{\prime}_{i}= (s_{i})_{\vert S}\circ p_{S}$,  $h_{i}=(t_{i})_{\vert S^{\bot}} p_{S^{\bot}}$ and $h^{\prime}_{i}=(s_{i})_{\vert S^{\bot}} \circ p_{S^{\bot}}$ where $p_{S}:T\to S$ and $p_{S^{\bot}}:T\to S^{\bot}$ are the projections.
We have $f_{i}\circ \rho=\beta_{i}$, $f^{\prime}_{i}\circ \rho=\gamma_{i}$, $g_{i}\circ \rho=\tilde\beta_{i}$, $g^{\prime}_{i}\circ \rho=\tilde\gamma_{i}$  and $f_{i}=g_{i}h_{i}$ and $f^{\prime}_{i}=g_{i}^{\prime}h^{\prime}_{i}$.
Therefore $\rho(g)$ is diagonal, for every $g \in G$. 

Suppose that $(\beta_{J}\gamma_{L})\lfloor_{\Gamma}=1$ for some $J, L \subseteq \{1,\ldots,m\}$ and consider 
\[f_{J}f^{\prime}_{L}=g_{J}g_{L}^{\prime}h_{J}h^{\prime}_{L}.
\]
Then, since $\rho(\Gamma)\subset S^{\bot}$, we have $(g_{J}g_{L}^{\prime})_{\vert \rho(\Gamma)}=1$.
Hence 
\[1=(\beta_{J}\gamma_{L})\lfloor_{\Gamma}=(f_{J}f^{\prime}_{L})_{\rho(\Gamma)}=(h_{J}h_{L}^{\prime})_{\rho(\Gamma)}.
\]
Since $S^{\bot}$ is the Zariski closure of $\rho(\Gamma)$, we get $h_{J}h_{L}^{\prime}=1$ and consequently $f_{J}f^{\prime}_{L}=g_{J}g_{L}^{\prime}$.
Hence if  $(\beta_{J}\gamma_{L})\lfloor_{\Gamma}=1$, we obtain 
\[\beta_{J}\gamma_{L}=\tilde\beta_{J}\tilde\gamma_{L}
\]
and thus the lemma follows.
\end{proof}

By using the previous  Lemma, Theorem \ref{dolbb}, and the  inclusions $B^{\ast,\ast}_{\Gamma}\subset\bigwedge (\tilde {\frak g}^{1,0})^{\ast}\otimes (\tilde {\frak g}^{0,1})^{\ast} \subset A^{\ast,\ast}(G/\Gamma)$, we obtain the isomorphism
\[H^{\ast,\ast}_{ \overline \partial}(G/\Gamma)\cong H_{ \overline \partial}^{\ast,\ast}\left(\bigwedge (\tilde {\frak g}^{1,0})^{\ast}\otimes (\tilde {\frak g}^{0,1})^{\ast}\right ).
\]

Since $\tilde\beta_{i}$ and $\tilde \gamma_{i}$ are unitary,
there exist holomorphic characters $\delta_{i}$  such that
\[\frac{\bar\delta_{i}}{\delta_{i}}=\tilde \beta_{i}\tilde \gamma _{i}.
\]
Then we can prove the following
\begin{lemma}
Let $\rho^{\prime}:G\to GL(\frak g_{\C})$ such that $\rho^{\prime}(g)X_{i}=X_{i}$, $\rho^{\prime}(g)\bar X_{i}= \bar X_{i}$, 
$\rho^{\prime}(g)Y_{i}=\delta_{i}Y_{i}$ and $\rho^{\prime}(g)\bar Y_{i}=\delta_{i} \bar Y_{i}$.
Then we have $\rho^{\prime}(g)\in\Aut^{d}_{\frak n_{\C}}(\frak g_{\C})$.
\end{lemma}

\begin{proof}
Consider the equations
\[t_{[m]}^{I}s_{[m]}^{J}=t_{[m]}^{K}s_{[m]}^{L}
\]
defining the algebraic group $\Aut^{d}_{\frak n_{\C}}(\frak g_{\C})$,
as in the proof of Lemma \ref{Au}. 
Since $S$ is a sub-torus of $T={\mathcal A} (\rho (G))\subset \Aut^{d}_{\frak n_{\C}}(\frak g_{\C})$,
we have
\[\tilde\beta_{[m]}^{I} \tilde\gamma_{[m]}^{J}=\tilde\beta_{[m]}^{K} \tilde\gamma_{[m]}^{L}.\]
By the complex conjugation,
we also get
\[\bar\alpha_{[m]}^{I} \alpha_{[m]}^{J}=\bar \alpha_{[m]}^{K} \alpha_{[m]}^{L}\]
and this implies
\[\gamma_{[m]}^{I} \beta_{[m]}^{J}=\gamma_{[m]}^{K} \beta_{[m]}^{L}\]
and so
\[\tilde\gamma_{[m]}^{I} \tilde\beta_{[m]}^{J}=\tilde\gamma_{[m]}^{K} \tilde\beta_{[m]}^{L}.\]
Thus we have
$$
\tilde\beta_{[m]}^{I}\tilde\gamma_{[m]}^{I}\tilde\beta_{[m]}^{J}\tilde\gamma_{[m]}^{J}
=\tilde\beta_{[m]}^{K}\tilde\gamma_{[m]}^{K}\tilde\beta_{[m]}^{L}\tilde\gamma_{[m]}^{L}.
$$
By the construction of $\delta_{i}$, we obtain
\[\frac{\bar\delta_{[m]}^{I}}{\delta_{[m]}^{I}}
\frac{\bar\delta_{[m]}^{J}}{\delta_{[m]}^{J}}
=\frac{\bar\delta_{[m]}^{K}}{\delta_{[m]}^{K}} \frac{\bar\delta_{[m]}^{L}}{\delta_{[m]}^{L}}.
\]
Since $\delta_{i}$ is holomorphic, the following relation holds
\[\delta_{[m]}^{I} \delta_{[m]}^{J}=\delta_{[m]}^{K} \delta_{[m]}^{L}.\]
Therefore $\rho^{\prime}(g)\in \Aut^{d}_{\frak n_{\C}}(\frak g_{\C})$.
\end{proof}
Take 
\[\begin{array}{l}
{\breve {\frak g}}^{1,0} =  {\rm {span}} <X_1, \ldots X_n,   \tilde \beta_1^{-1} \delta_{1}^{-1} Y_1, \ldots,   \tilde \beta_m^{-1}\delta_{m}^{-1} Y_m >,\\[4 pt]
{ \breve{\frak g}}^{0,1}  = {\rm {span}} <\overline X_1, \ldots  \overline X_n,    \tilde \gamma_1^{-1}\delta^{-1}_{1}\overline Y_1, \ldots,    \tilde \gamma_m^{-1}\delta^{-1}_{m}\overline Y_m >,
\end{array}\]
Then ${\breve {\frak g}}^{1,0}\oplus {\breve{\frak g}}^{0,1} $ is a complex Lie algebra and
 $\overline{\tilde{\beta_{i}}\delta_{i}}=\tilde \gamma _{i}\delta_i$ yields
\[\overline{{\breve {\frak g}}^{1,0}}= { \breve{\frak g}}^{0,1} .
\]
Since ${\breve {\frak g}}^{1,0}$ is closed under bracket,
 the complex Lie algebra ${\breve {\frak g}}^{1,0}\oplus {\breve{\frak g}}^{0,1} $ has a real form $\breve\g$ endowed with a integrable complex structure $\breve J$.
Consider the cochain complex
\[\bigwedge ({\breve {\frak g}}^{1,0})^{\ast}\otimes  ({\breve{\frak g}}^{0,1})^{\ast}.
\]
Since $\delta_{i}$ is holomorphic, we have a cochain complex isomorphism
\[\bigwedge ({\breve {\frak g}}^{1,0})^{\ast}\otimes  ({\breve{\frak g}}^{0,1})^{\ast}\cong \bigwedge (\tilde {\frak g}^{1,0})^{\ast}\otimes (\tilde {\frak g}^{0,1})^{\ast} .
\]
Thus $H_{ \overline \partial}^{*,*} (G/\Gamma) \cong H^{*,*}_{ \overline \partial } ( \breve {\frak g},  \breve J)$ and we can state the following theorem.
\begin{theorem}\label{semidirect_prod}
Let $G$ be a simply connected solvable Lie group  with a left-invariant complex structure $J$, a lattice $\Gamma$ and the nilradical $N$.
In the same setting as in Section \ref{hol_Mostow}, we suppose Condition \ref{condd}.
Furthermore, we assume  $H_{\overline{\partial}}^{*,*} (N/\Gamma  \cap N) \cong H_{\overline{\partial}}^{*,*} (\frak n)$.

Then there exists a finite index subgroup $\tilde \Gamma$ of $\Gamma$ and a real  Lie algebra $\breve\g$ with a  complex structure $\breve J$   such that
$$H_{ \overline \partial}^{*,*} (G/ \tilde \Gamma) \cong H^{*,*}_{ \overline \partial } ( \breve {\frak g},  \breve J).$$
Moreover 
 \begin{enumerate}
\item If we suppose  furthermore that $\g$ is completely solvable, then we can choose $\breve\g$ as a completely solvable Lie algebra.

\item If we suppose  furthermore that $J$ is an abelian complex structure (i.e. ${\mathfrak
g}^{1,0}$ is abelian), then we can choose $\breve J$ as an abelian complex structure on  $\breve\g$.
\end{enumerate}

\end{theorem}

%\begin{cor}
%Under the same assumptions as in  Theorem \ref{semidirect_prod},
%we suppose  furthermore that $\g$ is completely solvable.
%Then   there exists a finite index subgroup $\tilde \Gamma$ of $\Gamma$ and  a completely solvable real  Lie algebra $\breve\g$ with a  complex structure $\breve J$   such that
%$$H_{ \overline \partial}^{*,*} (G/ \tilde \Gamma) \cong H^{*,*}_{ \overline \partial } ( \breve {\frak g},  \breve J).$$
%\end{cor}

\begin{proof}[Proof of the assertion (1)]

Since $\alpha_{i}=\bar\alpha_{i}$, we have $\beta_{i}=\gamma_{i}$ and hence $\tilde\beta_{i}=\tilde \gamma_{i}$.
We have
$\delta_{i}=\tilde\alpha_{i}\overline{\tilde\beta_{i}}$ 
for some real  character $\tilde \alpha_{i}$ and thus 
\[\begin{array}{l}
{\breve {\frak g}}^{1,0} =  {\rm {span}} <X_1, \ldots X_n,  \tilde \alpha_{1}^{-1} Y_1, \ldots,   \tilde \alpha_{m}^{-1} Y_m >,\\[4 pt]
{ \breve{\frak g}}^{0,1}  = {\rm {span}} <\overline X_1, \ldots  \overline X_n,     \tilde \alpha_{1}^{-1} Y_1, \ldots,   \tilde \alpha_{m}^{-1} Y_m >.
\end{array} 
\]
\end{proof}

\begin{proof}[Proof of the assertion (2)]
We have $$[J X, J Y] = [X, Y], \quad \forall X, Y \in
{\mathfrak g}\, ,
$$
or equivalently 
$$d ({\cg}^{1,0})^{\ast}  \subset
  ({\cg}^{1,1})^{\ast}\, .$$
 We have
$$
\ad_X (JY) = - \ad_{JX} Y, \quad \forall X, Y \in \frak g.
$$
If $Z  \in \frak g^{1,0}$, i.e.  if $J Z = \sqrt{-1} Z$, we get
$$
(\ad_X + \sqrt{-1} J \ad_X) (Z) =0, \quad \forall X \in \frak g,
$$
and so in particular that $\ad_s$ vanishes on $\frak g^{1,0}$.
This implies that  $\alpha_{1},\dots, \alpha_{n}$ are anti-holomorphic.
Hence by the constructions of   $\gamma_{i}$, 
each $\gamma_{i}$ is trivial and each $\tilde\gamma_{i}$ is also trivial.

Since $\tilde \beta_i$ and $\tilde \gamma_i$ are unitary,  we have holomorphic characters $\delta_i$ such that
$$
\frac{\overline \delta_i} {\delta_i} = \tilde \beta_i .
$$
We can show  that the  complex structure $ \breve J$ on ${\breve {\frak g}}$ is abelian since
$$
{\breve {\frak g}}^{1,0} =  {\rm {span}} <X_1, \ldots X_n,   \overline \delta_{1}^{-1} Y_1, \ldots,    \overline \delta_{m}^{-1} Y_m >$$
and the $  \overline \delta_{i}^{-1}$ are anti-holomorphic.

\end{proof}

\begin{rem}
Note that since $J$ is abelian we  always have \cite{CF, CFGU}
$$
H^{*, *}_{\overline \partial} (\frak n) \cong  H^{*,*}_{\overline \partial} (N/\Gamma_N, J_N)\, ,
$$
where $N$ is the nilradical of $G$ and $J_N$ is the abelian complex structure induced by $J$ on $N$.
\end{rem}

\begin{ex}  We can apply  Theorem  \ref{semidirect_prod} to the Lie group  $G= \C \ltimes_{\phi}   \C^2$  with
$$
\phi (z_1) = \left ( \begin{array}{cc} e^{\overline z_1} & 0\\ 0 & e^{- \overline z_1} \end{array} \right ).
$$
Consider on $G$ the abelian complex structure $J$ defined by 
$$
\cg^{1,0}= {\rm {span}} \left \langle \frac {\partial}{\partial z_1}, e^{\bar z_1} \frac {\partial}{\partial z_2}, 
e^{-\bar z_1} \frac {\partial}{\partial z_3}\right \rangle\, .
$$
There is  a lattice  $\Gamma$ in $G$  of the form $(\Z + 2 \pi \sqrt{-1} \Z) \ltimes \Gamma'$ with $\Gamma'$ lattice in $N = \C^2$.
 In this case we have $\delta_i = \overline \alpha_i$ and $\tilde \alpha_i = \alpha_i$. Since
 $$
 \alpha_1 = e^{\overline z_1}, \quad  \alpha_2 = e^{- \overline z_1}
 $$
 we obtain that  
 $$
  \breve {\frak g}^{1,0} = {\rm {span}}  <X_1, Y_1, Y_2>= {\rm {span}}  \left \langle \frac {\partial}{\partial z_1},  \frac {\partial}{\partial z_2}, 
\frac {\partial}{\partial z_3}\right \rangle\, .
  $$
  Therefore  the modification $( \breve {\frak g},  \breve J)$ of $({\frak g}, J)$ is the abelian Lie algebra $\C^3$.
\end{ex}

In the case of a \emph{complex parallelizable compact solvmanifold},  
the assumptions of Theorem \ref{semidirect_prod} hold and hence we can show the following.

\begin{cor}\label{parall}
Let $G$ be a simply connected complex solvable Lie group with a lattice $\Gamma$.
Then there exists a finite index subgroup $\tilde \Gamma$ of $\Gamma$ and  a complex  Lie algebra $\breve\g$ such that
$$H_{ \overline \partial}^{*,*} (G/\tilde \Gamma)\cong \bigwedge^*  \breve {\frak g}^{1,0} \otimes  H^*( \breve {\frak g}^{0,1})\, .$$
\end{cor}
\begin{proof}
In case $\g$ is a complex solvable Lie algebra, $\alpha_{i}$ is holomorphic and so $\beta_{i}$ and $\tilde\beta_{i}$ are trivial.
Hence we have $\bar\delta_{i}=\delta_{i}\tilde \gamma _{i}$ and
\[\begin{array}{l}
{\breve {\frak g}}^{1,0} =   {\rm {span}} <X_1, \ldots X_n,    \delta_{1}^{-1} Y_1, \ldots,   \delta_{m}^{-1} Y_m >,\\[4 pt]
{ \breve{\frak g}}^{0,1}  =  {\rm {span}} <\overline X_1, \ldots  \overline X_n,    \bar\delta^{-1}_{1}\overline Y_1, \ldots,     \bar\delta^{-1}_{m}\overline Y_m >.
\end{array}\]
Thus  the corollary follows.
\end{proof}

\begin{rem}
In general, it is not true that there exists a simply connected solvable Lie group $\breve G$ with the Lie algebra $\breve \g$ containing $\tilde \Gamma$ as a lattice.
For example, let $G=\C\ltimes_{\phi} \C^{2}$ such that \[\phi(x+\sqrt{-1}y)=\left(
\begin{array}{cc}
e^{x+\sqrt{-1}y}& 0  \\
0&    e^{-x-\sqrt{-1}y}  
\end{array}
\right)\]
with a lattice $\Gamma=(a\Z+2\pi\sqrt{-1})\ltimes \Gamma^{\prime\prime}$.
Then we have $H^{\ast,\ast}(G/\Gamma)\cong \bigwedge \C^{3}$ and hence we get $\breve \g\cong \C^{3}$.
But no lattice in $G$ embeds in  $\breve G=\C^{3}$  with the Lie algebra $\breve \g$.
\end{rem}

\smallskip

{\it Acknowledgments.}  We thank the referees for their  comments and help in improving the contents of this paper.

\end{document}